\documentclass[letter,11pt,leqno]{amsart}
\usepackage[margin=1 in,bindingoffset=0cm,nofoot]{geometry}

\usepackage{amsmath, amssymb, amsthm, stmaryrd}
\usepackage{euscript}
\usepackage{verbatim}
\usepackage{graphicx}
\numberwithin{equation}{section}

\newcommand{\riem}{\mathcal{R}}
\newcommand{\riemdddd}[4]{\riem_{#1 #2 #3 #4}}
\newcommand{\riemuddd}[4]{\riem^{#1}_{\phantom{#1} #2 #3 #4}}

\newcommand{\riemdudd}[4]{\riem^{\phantom{#1} #2}_{#1 \phantom{#2} #3 #4}}
\newcommand{\riemddud}[4]{\riem^{\phantom{#1 #2} #3}_{#1 #2 \phantom{#3} #4}}
\newcommand{\riemdddu}[4]{\riem^{\phantom{#1 #2 #3} #4}_{#1 #2 #3}}
\newcommand{\riemudud}[4]{\riem^{#1 \phantom{#2} #3}_{\phantom{#1} #2 \phantom{#3} #4}}

\newcommand{\riemcons}{\mathcal{K}}

\newcommand{\triem}{{\mathcal{R}^\Sigma}}
\newcommand{\triemdddd}[4]{\triem_{#1 #2 #3 #4}}

\newcommand{\weyl}{\mathcal{W}}
\newcommand{\weyldddd}[4]{\weyl_{#1 #2 #3 #4}}

\newcommand{\weyludud}[4]{\weyl^{#1 \phantom{#2} #3}_{\phantom{#1} #2 \phantom{#3} #4}}

\newcommand{\ric}{\mathcal{R}\mathrm{ic}}
\newcommand{\ricdd}[2]{\ric_{#1 #2}}
\newcommand{\ricud}[2]{\ric^{#1}_{\phantom{#1} #2}}
\newcommand{\ricdu}[2]{\ric^{\phantom{#1} #2}_{#1}}

\newcommand{\tdddd}[5]{\mathcal{#1}_{#2 #3 #4 #5}}

\newcommand{\sff}{S}
\newcommand{\sffdd}[2]{S_{#1 #2}}
\newcommand{\sffud}[2]{S^{#1}_{\phantom{#1} #2}}

\newcommand{\tgrad}[1]{\nabla^{\Sigma}_{#1}}

\newcommand{\kronecker}[2]{\delta^{#1}_{\phantom{#1}#2}}

\def\rq{\noindent\textbf{Remark: }}

\newtheorem{maintheorem}{Theorem}

\newtheorem{theorem}{Theorem}[section]
\newtheorem{prop}[theorem]{Proposition}

\newtheorem{lemma}[theorem]{Lemma}

\newcommand{\Mbar}{\overline{M}} 
\newcommand{\Minf}{M(\infty)} 
\newcommand{\bR}{\mathbb{R}}
\newcommand{\bB}{\mathbb{B}}
\newcommand{\bS}{\mathbb{S}}

\newcommand{\bN}{\mathbb{N}}
\newcommand{\bH}{\mathbb{H}}
\newcommand{\cA}{\mathcal{A}}

\newcommand{\gbar}{\overline{g}}

\renewcommand{\hbar}{\overline{h}}
\newcommand{\Gambar}{\overline{\Gamma}}
\newcommand{\Gamtil}{\widetilde{\Gamma}}
\newcommand{\gtil}{\widetilde{g}}

\newcommand{\grd}{\mathring{g}}

\newcommand{\pd}[2]{\frac{\partial #1}{\partial #2}}
\newcommand{\caution}{\textbf{Caution: \hspace{ 0.25 cm }}}

\newcommand{\Wsff}{W}

\newcommand{\Wsffud}[2]{\Wsff^{#1}_{\phantom{#1} #2}}

\begin{document}
\title[ALH metrics]{Conformal compactification of asymptotically locally hyperbolic metrics}
\author[Eric Bahuaud and Romain Gicquaud]{Eric Bahuaud and Romain Gicquaud}
\date{\today}
\address{Current address: \newline
Institut de Math\'ematiques et de Mod\'elisation de Montpellier \newline
UMR 5149 CNRS - Universit\'e Montpellier II \newline
Case Courrier 051 - Place Eug\`ene Bataillon \newline
34095 Montpellier, France \newline }
\email{ebahuaud@msri.org}
\email{Romain.Gicquaud@math.univ-montp2.fr}
\begin{abstract}
In this paper we study the extent to which conformally compact asymptotically hyperbolic metrics may be characterized intrinsically.  Building on the work of the first author in \cite{Bahuaud}, we prove that decay of sectional curvature to $-1$ and decay of covariant derivatives of curvature outside an appropriate compact set yield H\"older regularity for a conformal compactification of the metric.  In the Einstein case, we prove that the estimate on the sectional curvature implies the control of all covariant derivatives of the Weyl tensor, permitting us to strengthen our result.
\end{abstract}

\maketitle
\tableofcontents

\section{Introduction}
\label{section:introduction}

The study of complete Riemannian manifolds remains a lively and important topic of research.  Constant sectional curvature spaces are well understood: these spaces are quotients of simply connected models (either euclidean space, the round sphere or hyperbolic space) by a discrete group of isometries.  In an effort to understand more general Riemannian metrics on non-compact manifolds, it is natural to study metrics that approach one of the constant curvature models in some sense.  From a physical viewpoint these metrics are of interest in their own right as they represent natural Cauchy surfaces for isolated systems in general relativity.  Much is known about the structure at infinity of asymptotically flat manifolds, see for example \cite{Bartnik} and \cite{Bando} and references therein.  The purpose of this paper is to compare the difference between `classical' AH metrics and the (more natural) notion of asymptotically locally hyperbolic metrics.\\

We begin with a very heuristic idea of our approach.  We want our definition to be intrinsic, i.e. not to depend on choices of coordinates or the a priori existence of a manifold compactification. We first review a few basic facts. The Poincar\'e model of hyperbolic space $\bH^{n+1}$ is the the open unit ball $\bB^{n+1} \subset \bR^{n+1}$ equipped with the metric
\[ h = \frac{4}{(1-|x|^2)^2} \left((dx^1)^2 + \ldots + (dx^{n+1})^2\right).\]
The sectional curvature of $h$ for all two planes is $-1$.  Consequently, the hyperbolic metric is Einstein: $\ric_h = -nh$.  It is well known that a good description of the geometry of $\bH^{n+1}$ involves the boundary sphere at infinity.  This sphere may be described purely intrinsically as follows: define an equivalence relation on the set of geodesic rays parameterized by arc-length by saying $\sigma$ and $\tau$ are \textit{asymptotic} if $d_h(\sigma(t), \tau(t))$ remains bounded as $t \rightarrow +\infty$.  Denote the set of equivalence classes by $\bB(\infty)$.  One can show that given any point $p \in \bB^{n+1}$, $\bB^{n+1}(\infty)$ is in bijection with the unit sphere $S_p \bB^{n+1} \subset T_p \bB^{n+1}$ by a rescaled exponential map.  We obtain the \textit{geodesic compactification} $\overline{\bB^{n+1}} = \bB^{n+1} \bigcup \bB^{n+1}(\infty)$ by declaring this map to be a homeomorphism. Declaring this map to be a diffeomorphism gives a smooth structure on $\overline{\bB^{n+1}}$; in the case of hyperbolic space the smooth structure is independent of $p$. We note that the procedure outlined here was extended to arbitrary manifolds of nonpositive curvature by \cite{EberleinONeill}. In general one only expects the topological structure to be independent of $p$ (but see also \cite{AndersonSchoen}, \cite{BahuaudMarsh}).\\

Observe in the above compactification construction that we used the exponential map from a point to achieve the diffeomorphism. We could have well replaced the exponential map from a point $p$ with the normal exponential map from a sphere centered at $p$. This justifies the following definition: an \textit{essential subset}\label{pg:es} $K^{n+1}$ of a complete Riemannian manifold $M^{n+1}$ is a compact embedded submanifold with boundary $Y^n = \partial K$ such that $Y$ is convex with respect to the outward unit normal and the normal exponential map $E: Y \times [0, \infty) \longrightarrow M \setminus \mathring{K}$ is a diffeomorphism.\\

One of the most important established models of asymptotically hyperbolic metrics is based on conformal compactifications, which we now describe.  There are numerous other notions of asymptotically hyperbolic metrics, see \cite{ChruscielHerzlich} and \cite{HerzlichMass} for examples similar in spirit to the asymptotically flat case and relations to positive mass theorems.  Suppose $(M,g)$ is a noncompact Riemannian $(n+1)$-manifold that is the interior of a compact manifold with boundary $\Mbar$.  For $k \in \bN_0, \alpha \in [0,1]$, the metric $g$ is \textit{$\mathcal{C}^{k,\alpha}$ conformally compact} if there exists a defining function $\rho$ for the boundary such that $\gbar = \rho^2 g$ extends to a $\mathcal{C}^{k,\alpha}$ metric on $\Mbar$.  Such a metric induces a conformal class on the boundary $\partial M$, called the \textit{conformal infinity} of $g$.\\

Straightforward calculations yield that if $g$ is at least $\mathcal{C}^2$ conformally compact then the sectional curvatures in $M$ satisfy
\begin{equation} \label{AHdecay} 
\sec = -|d\rho|^2_{\gbar} + O(\rho)~\text{near $\partial M$}. 
\end{equation}
If $|d\rho|^2_{\gbar} = 1$ on $\partial M$, then the sectional curvatures of $M$ approach $-1$ near $\partial M$.  This justifies the following definition.  The metric $g$ is \textit{asymptotically hyperbolic} if $g$ is conformally compact and $|d\rho|^2_{\gbar} = 1$ on $\partial M$.  The classical setting typically requires at least a $\mathcal{C}^2$ conformal compactification.  As any two defining functions for $\partial M$ differ by a multiplication by a positive function, this definition is easily seen to be independent of $\rho$.  When $g$ is additionally an Einstein metric, i.e. $\ric_g = - n g$ (the `cosmological constant' being determined by the asymptotic value of the sectional curvature), the sectional curvatures of $g$ satisfy an improved decay estimate, i.e. \label{AHEDecayPage}
\begin{equation} \label{AHEdecay} 
\sec = -1 + O(\rho^2). 
\end{equation}
This is a consequence of the transformation law for the Ricci tensor under a conformal change of metric which proves that the first order correction for the sectional curvature must vanish for an Einstein manifold, see e.g. \cite{AndersonEinstein}.\\

Conformally compact metrics have proved to be important in Riemannian and conformal geometry in no small part due to the work of Fefferman and Graham \cite{FG} and to Maldacena's AdS/CFT correspondence (see for example \cite{AdS-CFT}).  The basic outline is to relate the Riemannian geometry of the Einstein metric $g$ to the conformal geometry of the conformal infinity $\partial M$.  A lot of work has been dedicated to the existence and regularity questions of these metrics, see e.g. \cite{GrahamLee}, \cite{Lee}, \cite{AndersonBoundary}, \cite{CDLS} and the references therein.  Such manifolds also appear in other contexts such as general relativity where they are good candidates for Cauchy surfaces in asymptotically simple space-times.  We refer the reader to \cite{AnderssonChrusciel} and \cite{Gicquaud} for more details.  In this setting, the conformal infinity is not given a priori and it is a natural question to wonder to what extent the boundary at infinity can be reconstructed.  In particular, the regularity of the compactified metric is an important ingredient in applying elliptic theory in these spaces, see \cite{Lee} for example.\\

In an earlier paper \cite{Bahuaud}, the first author began to study to what extent conformally compact AH metrics can be characterized intrinsically. We review this result. Conformally compact metrics possess essential subsets (just take $K = \{ \rho \geq \epsilon\}$ for $\epsilon$ sufficiently small) so this definition provides a good departure point for our study.  
The main result of \cite{Bahuaud} is
\begin{theorem} \label{thm:EricThesisResult} Suppose $(M,g)$ is a complete noncompact Riemannian manifold and $K$ is an essential subset.  Let $r(x) = dist_g(x, K)$.  Assume further that
\begin{equation*} 
\sec(  M \setminus \mathring{K} ) < 0, \text{and}
\end{equation*}
\begin{equation*} 
\sec( M \backslash K ) = -1 + o(1), \text{ and }
\end{equation*} 
\begin{equation*} 
|\nabla_g \riem|_g = O( e^{-\omega r} ), \; \text{for some} \; \omega > 1.
\end{equation*}
Then $\Mbar = M \cup M(\infty)$ is a topological manifold with boundary endowed with a $\mathcal{C}^{1,1}$ structure independent of $K$.  Further $\gbar := e^{-2r} g$ extends to a $\mathcal{C}^{0,1}$ metric on $\Mbar$, i.e. $g$ is $\mathcal{C}^{0,1}$ conformally compact.
\end{theorem}

As an example, all of the assumptions above hold sufficiently close to the boundary of a smoothly conformally compact Einstein metric.  We remark \label{remark:curv} that assumption on the covariant derivative of curvature implies that sectional curvature estimate enjoys the same rate of decay, i.e. in fact $\sec( M \backslash K ) = -1 + O( e^{-\omega r} )$.  This is easily seen by integrating the components of $\riem+\riemcons$, where $\riemcons$ denotes the constant curvature tensor: $$\riemcons_{abcd} = g_{ad} g_{bc} - g_{ac} g_{bd},$$ with respect to a parallel frame along normal geodesics emanating from $K$.

We briefly describe the proof of this theorem.  Mimicking the classical geodesic compactification described above, it was proved in \cite{BahuaudMarsh} that $\Minf$ is in bijection with the boundary $Y$ by a rescaled exponential map, and that there is a natural topology on the geodesic compactification $\Mbar := M \cup M(\infty)$.  Further, it was proved in \cite{Bahuaud} that the asymptotic curvature pinching implies that $\Mbar$ has the structure of a $\mathcal{C}^{0,1}$ manifold independent of $K$.  Finally in order to prove that $\gbar$ is a Lipschitz metric, we take derivatives of the Riccati equation for the shape operator and metric of constant $r$ level sets in appropriate coordinates, and analyze the resulting system. Then the rough $\mathcal{C}^{0, 1}$-structure of $\Mbar$ can be improved to a $\mathcal{C}^{1,1}$-structure independent of $K$ using a trick of Calabi-Hartman \cite{CalabiHartman}.

The purpose of the present paper is twofold.  We first extend the result of \cite{Bahuaud} to obtain complete understanding of how the rate of curvature decay influences the regularity of the conformal compactification, and we explain how further regularity can be obtained by assuming appropriate decay of $|\nabla^2 \riem|$.  In particular we prove
\begin{maintheorem} \label{thm:A-part1} Suppose $(M,g)$ is a complete noncompact Riemannian manifold and $K$ is an essential subset.  Let $r(x) = dist_g(x, K)$.  Assume further that
\begin{equation}  \label{NSC} \tag{NSC}
\sec(  M \setminus \mathring{K} ) < 0, \text{and}
\end{equation}
\begin{equation} \label{AH0} \tag{AH0}
|\riem + \riemcons|_g = O(e^{-ar}),
\end{equation}
\begin{equation} \label{AH1} \tag{AH1}
|\nabla_g \riem |_g = O(e^{-ar}),
\end{equation}
Then:
\begin{itemize}
\item If $0 < a < 1$, $\Mbar = M \cup M(\infty)$ is endowed with a $\mathcal{C}^{1,a}$ structure independent of $K$, and $\gbar := e^{-2r} g$ extends to a $\mathcal{C}^{0,a}$ metric on $\Mbar$.
\item If $a = 1$, $\Mbar = M \cup M(\infty)$ is endowed with a $\mathcal{C}^{1,b}$ structure independent of $K$, and $\gbar := e^{-2r} g$ extends to a $\mathcal{C}^{0,b}$ metric on $\Mbar$, for every $b \in (0,1)$.
\item If $a > 1$, $\Mbar = M \cup M(\infty)$ is endowed with a $\mathcal{C}^{1,1}$ structure independent of $K$, and $\gbar := e^{-2r} g$ extends to a $\mathcal{C}^{0,1}$ metric on $\Mbar$.
\end{itemize}
\end{maintheorem}

Note that this Theorem is sharp in the case $a=1$.  In \cite{Bahuaud} the first author provided an example of a metric which satisfies \eqref{AH0} and \eqref{AH1} for $a=1$ but with no Lipschitz conformal compactification.

\begin{maintheorem} \label{thm:A-part2} 
Given all of the hypothesis of Theorem \ref{thm:A-part1}, assume additionally 
\begin{equation} \label{AH2} \tag{AH2}
|\nabla_g^2 \riem |_g = O(e^{-ar}),
\end{equation}
\begin{itemize}
\item If $1 < a < 2$, $\Mbar = M \cup M(\infty)$ is endowed with a $\mathcal{C}^{2,a-1}$ structure independent of $K$, and $\gbar := e^{-2r} g$ extends to a $\mathcal{C}^{1,a-1}$ metric on $\Mbar$.
\item If $a = 2$, $\Mbar = M \cup M(\infty)$ is endowed with a $\mathcal{C}^{2,b}$ structure independent of $K$ and $\gbar := e^{-2r} g$ extends to a $\mathcal{C}^{1,b}$ metric on $\Mbar$, for every $b \in (0,1)$.
\end{itemize}
\end{maintheorem}

The proof of this theorem mimics its counterpart in \cite{Bahuaud}.  We do not pursue analysis for faster decay because of rigidity results for asymptotically hyperbolic metrics; see \cite{ShiTian} for example.  Note that while we present our results in terms of H\"older-type estimates, it easy to obtain $W^{2,p}$-Sobolev estimates for the compactified metric $\gbar$, however the Sobolev embedding theorem applied to these estimates does not yield optimal H\"older regularity.\\

The second purpose of this paper is strengthen our results significantly in the case that $g$ is Einstein, i.e. $\ric_g=-ng$.  As we will deal only with manifolds whose curvature tends to $-1$ at infinity, we assume implicitly the normalization $\ric_g = -n g$. We prove 

\begin{maintheorem} \label{thm:B} Let $(M,g)$ be a complete noncompact Riemannian manifold containing an essential subset such that (AH0) holds for an arbitrary $a>0$ and $g$ is Einstein.  Then $|\nabla_g^{(j)} \riem |_g = O(e^{-ar})$ holds for all $j \geq 1$.  
\end{maintheorem}

In particular, this theorem is valid for any conformally compact Einstein manifold.  This confirms the naive idea for an Einstein metric that anything better than $\mathcal{C}^2$-boundary regularity is non-local data and cannot be detected by the behaviour of curvature quantities at infinity.  Applying this theorem and Theorems \ref{thm:A-part1}, \ref{thm:A-part2} we obtain immediately that an Einstein metric $g$ with sectional curvature decay $|\riem + \riemcons|_g = O(e^{-2r})$ is $\mathcal{C}^{1,b}$ conformally compact for every $0 < b < 1$.\\

We remark that whereas the proof of Theorems \ref{thm:A-part1}, \ref{thm:A-part2} use ODE analysis to obtain estimates, the proof of Theorem \ref{thm:B} uses elliptic PDE theory.  The Einstein condition allows us to construct appropriate harmonic coordinate balls where the metric $g$ and its derivatives are appropriately controlled.  From a standard formula for the Laplacian of the curvature $4$-tensor we derive an elliptic equation for the Weyl curvature tensor.  We then use elliptic theory to conclude that all derivatives of this tensor decay to the same order.  In a forthcoming paper, the authors plan to study the case of ALH Einstein manifolds in greater detail.\\

This paper is organized as follows.  In Section \ref{section:notation}, we fix notation and prove the basic shape operator and metric estimates to obtain our first manifold compactification.  In Section \ref{section:riccati-analysis} we study the Riccati system for the metric and shape operator in the same spirit as \cite{Bahuaud} and prove Theorems \ref{thm:A-part1} and \ref{thm:A-part2}.  In Section \ref{section:einstein} we study the Einstein case in detail and prove Theorem \ref{thm:B}.\\

\noindent\textit{Acknowledgments.} The authors are grateful to Erwann Delay, Marc Herzlich, Jack Lee and Rafe Mazzeo for useful discussions and support.  We also thank Piotr Chru\'sciel and Michael Anderson for their interest in this work.
% Notation
\section{Background and Notation}
\label{section:notation}

In this section we fix notation and recall a few facts.  Throughout the paper $M$ denotes a complete non-compact smooth Riemannian manifold and $K$ an essential subset of $M$:  recall that this means $K^{n+1}$ is a compact embedded submanifold with boundary $Y^n = \partial K$ such that $Y$ is convex with respect to the outward unit normal and the normal exponential map $E: Y \times [0, \infty) \longrightarrow M \setminus \mathring{K}$ is a diffeomorphism.  One sufficient condition to imply the existence of an essential subset was given in \cite{BahuaudMarsh}: if $K$ is totally convex in $M$ and $\sec(M \setminus K) < 0$ then $K$ is an essential subset (recall $K$ is totally convex if for all $p, q \in K$ and \textit{any} geodesic curve $\gamma: [0,1] \rightarrow M$ with $\gamma(0) = p$, $\gamma(1) = q$ then $\gamma([0,1]) \subset K$).  Conformally compact metrics possess essential subsets: if $\rho$ is a defining function then $K = \rho^{-1}([\epsilon; \infty))$ is an essential subset for small enough $\epsilon > 0$ due to convexity properties of the function $r = - \log \rho$ near the boundary (the proof is similar to Lemma \ref{LmClosedGeod}).

In light of the diffeomorphism $Y \times [0, \infty) \approx M \setminus \mathring{K}$ and the fact that $r$ is the distance to $K$, we may decompose $g$ as
\[ g = dr^2 + g_r, \]
where $g_r$ is a one parameter family of metrics on $Y$.  We cover $Y$ with finitely many sufficiently small normal coordinate balls as in \cite{BahuaudMarsh}.  We label such coordinates $\{y^{\alpha}\}$, and extending such coordinates to be constant along the integral curves of $r$ provides Fermi coordinates on cylinders.  In such a cylinder the metric decomposes as
\[ g = dr^2 + g_{\alpha \beta}(y, r) dy^{\alpha} dy^{\beta}. \]
We use Greek indices \label{page:greek} (with the exception of $\rho$) to index directions along $Y$ and consequently these range from $1$ to $n$.  We use Latin indices to index directions in $M$ and these range from $0$ to $n$; we consistently use the subscript $0$ for the normal direction.

We will need to consider various curvature quantities for the metric restricted to constant $r$ slices.  We denote constant $r$ slices by $\Sigma_r$, often omitting the subscript.  Define the second fundamental form of $r$-level sets by $\sff(X,Z) = g( \nabla_X Z, -\partial_r )$ where $X, Z$ are tangent to $\Sigma_r$.  We denote the shape operator (a (1,1)-tensor) by the same symbol $\sff$.  Note that in Fermi coordinates the following relations are useful.
\[ \sffdd{\alpha}{\beta} = g( \nabla_{\partial_{\alpha}} \partial_{\beta}, -\partial_r) = g( \Gamma^0_{\alpha \beta} \partial_r , -\partial_r) = - \Gamma^0_{\alpha \beta} = \frac{1}{2} \partial_r g_{\alpha \beta}, \]
\[ \sffud{\beta}{\alpha} = g^{\beta \gamma} \sffdd{\gamma \alpha} = \frac{1}{2} g^{\beta \gamma} \partial_r g_{\gamma \alpha} = \Gamma^{\beta}_{0 \alpha}.\]
The sectional curvature of a two plane spanned by orthogonal unit vectors $X$ and $Z$ is given by $\riem(X,Z,Z,X)$.  Note that we denote the Weyl curvature by $\weyl$ and the constant curvature tensor by $\riemcons$.  In the case of an Einstein metric $\weyl = \riem + \riemcons$.  Throughout this paper we normalize the Einstein constant to be $-n$.

We collect some fundamental equations here for reference \cite{Petersen}: \label{pg:curveqns}
\begin{itemize}
\item \emph{Riccati equation}
\begin{equation} \label{eqn:Riccati}
\partial_{r} \sffud{\beta}{\alpha}  + \sffud{\beta}{\gamma} \sffud{\gamma}{\alpha}= -\riemdudd{0}{\beta}{\alpha}{0},
\end{equation}
\item \emph{Gauss equation}
\begin{equation} \label{TangentialCurvatureEq}
\riemdddd{\alpha}{\beta}{\gamma}{\delta}
	= \triemdddd{\alpha}{\beta}{\gamma}{\delta}
	+ \sffdd{\alpha}{\gamma} \sffdd{\beta}{\delta} - \sffdd{\alpha}{\delta} \sffdd{\beta}{\gamma},
\end{equation}
\item \emph{Codazzi-Mainardi equation}
\begin{equation}\label{NormalCurvatureEq}
\tgrad{\alpha} \sffdd{\beta}{\gamma} - \tgrad{\beta} \sffdd{\alpha}{\gamma} = - \riemdddd{\alpha}{\beta}{\gamma}{0},
\end{equation}
\item \emph{Evolution of the tangential metric under the geodesic flow}
\begin{equation}
\label{LieDerivativeMetricEq}
\partial_r g_{\mu\nu} = \mathcal{L}_{\partial_r} g_{\mu\nu} = 2 \sffdd{\mu}{\nu} (=2 g_{\mu\sigma} \sffud{\sigma}{\nu}).
\end{equation}
\end{itemize}

In what follows an inequality involving the shape operator of the form $\sffud{\beta}{\alpha} \geq c$  means that every eigenvalue of the shape operator $\sff$ is greater than or equal to $c$.  Inequalities involving a metric are to be interpreted as inequalities between quadratic forms.

We now outline our notation for order estimates.  If a tensor appears with subscripts in an order estimate, then an estimate of components in Fermi coordinates is implied, otherwise the tensor norm is implied.  For example a $2$-tensor $T$, the notation $T_{ij} = O(e^{-ar})$ means that the components of $T$ in Fermi coordinates satisfy the estimate.  As another example, the tensor norm quantity estimated in \eqref{AH0} implies the following estimate of components:
\[ (\riem + \riemcons)_{ijkl} = O( e^{(4-a)r} ).\]

Similarly,
\[ (\riem_{\partial r})_{\phantom{\beta}\alpha}^{\beta} + \kronecker{\beta}{\alpha} = \riemdudd{0}{\beta}{\alpha}{0} + \kronecker{\beta}{\alpha} = O(e^{-ar}),\]
\[ \riemdudd{0}{\beta}{\alpha}{\sigma} = O(e^{(1-a)r}), \; \riemdudd{\sigma}{\beta}{\alpha}{0} = O(e^{(1-a)r}).\]

We also document the following estimates derived from \eqref{AH1}:\label{pg:curvests}
\[ \nabla_{\mu} \riemdudd{0}{\beta}{\alpha}{0} = | (\nabla \riem)( \partial_r, dx^{\beta}, \partial_{\alpha}, \partial_r, \partial_{\mu})| = O(e^{(1-a)r}),\]
\[ \nabla_{0} \riemdudd{0}{\beta}{\alpha}{0} = O(e^{-ar}),\]
\[ \nabla_{\mu} \riemdudd{0}{\beta}{\alpha}{\sigma} = O(e^{(2-a)r}),\]
\[ \nabla_{\nu} \nabla_{\mu} \riemdudd{0}{\beta}{\alpha}{0}= O(e^{(2-a)r}).\]

We use the results of \cite{BahuaudMarsh} to compactify $M$ and obtain the first estimate of manifold regularity.  In order to do this we must prove metric estimates in an atlas of carefully chosen Fermi coordinates.  Following \cite{BahuaudMarsh} we cover $Y := \partial K$ by a \label{pg:refcov}reference covering of finitely many small open $g_Y$-normal coordinate balls $\{W_i\}$.   The $\{ W_i\}$ are chosen with sufficiently small radius chosen so that $g_{\alpha \beta}$ (transfered to $W$ by means of normal coordinates), the round metric $\grd$ on $\bS^n$ in normal coordinates and the flat metric on $W$ are comparable.  We show that an appropriate metric estimate holds on Fermi charts \label{choosingW} of the form $W \times [R, \infty)$, where $W \subset W_i$.  For $R>0$ sufficiently large, $g$ is then comparable to `comparison' hyperbolic metrics
\[ dr^2 + \sinh^2 (r \pm R) \; \grd, \]
and we may cite the compactification result Theorem 17 of \cite{BahuaudMarsh}.  See \cite{Bahuaud} for more details.

\subsection{The Riccati system} \label{sec:riccatisystem} 
As mentioned above, in a Fermi coordinate chart, the metric and shape operator satisfy the following system of differential equations that we refer to as the Riccati system:
\begin{equation} \label{Riccati.System}
\left\{\begin{aligned}
(\sffud{\beta}{\alpha})'  + \sffud{\beta}{\gamma} \sffud{\gamma}{\alpha}&= -(\riem_{\partial r})_{\phantom{\beta}\alpha}^{\beta}, \; \; \text{and}
  \\
g_{\alpha \beta} ' &= 2 \sffud{\gamma}{\alpha} g_{\gamma \beta},
\end{aligned}\right.
\end{equation}
where primes denote derivatives with respect to $r$ and $\riem_{\partial r}$ is the normal curvature operator that satisfies
$g( \riem_{\partial r} X, X) = \sec(\partial_r, X)$, for $X$ a $g$-unit vector.

We now prove estimates for the shape operator by analyzing the Riccati equation.  We begin by considering the following scalar differential inequality.
\begin{equation*} 
\left\{
\begin{aligned}
\lambda' + \lambda^2 &= 1 + O(e^{-ar}),  \\
\lambda(0) & > 0.
\end{aligned} \right.
\end{equation*}

\begin{lemma} \label{lemma:scal-riccati-est}  Suppose that $f \in L^\infty( [0,\infty) )$ such that there exists constants $\epsilon > 0$ and $J>0$ with
\[ \left\{
\begin{aligned}
 f &> \epsilon~\text{a.e.},\\
|f(r) - 1| &\leq J e^{-ar}~\text{a.e.},\end{aligned}\right. \]
where $a > 0$. Suppose further that $\lambda$ is a solution of the Riccati equation

\begin{equation}
\begin{aligned}
\lambda' + \lambda^2 &= f(r), \text{and } \nonumber \\
\lambda(0) & > 0.
\end{aligned}
\end{equation}
Then $\lambda$ is a positive Lipschitz function such that, for a positive constant $C = C(a, J, \lambda(0))$,
$$
\left\{
\begin{aligned}
 | \lambda - 1 | \leq C e^{-ar}\quad\text{if $a < 2$},\\
 | \lambda - 1 | \leq C (r+1) e^{-2r}\quad\text{if $a = 2$}\\
 | \lambda - 1 | \leq C e^{-2r}\quad\text{if $a > 2$}
\end{aligned}
\right.
$$
for all $r > 0$.
\end{lemma}

\begin{proof} The proof we present here is different from \cite{Bahuaud}.
Our first task is to prove that $\lambda$ is a positive function. To this end, we select $\mu > 0$ such that $2 \mu^2 < \epsilon$ and $\lambda(0)>\mu$. We prove that $\lambda > \mu$. Assume that there exists a $r>0$ such that $\lambda(r)\leq\mu$, let $R=\inf\{r\in\bR^*_+|\lambda(r)\leq\mu\}$. The assumption $\lambda(0)>\mu$ and the continuity of $\lambda$ imply that $R > 0$, $\lambda(R) = \mu$ and that $\lambda(r)>\mu$ for all $r < R$.  Select $h > 0$ small enough such that $\lambda^2(r) \leq \frac{3}{2} \mu^2$ for all $r \in [R-h, R]$. Then

$$
\begin{aligned}
0 & \geq \lambda(R) - \lambda(R-h)\\
  & \geq \int_{R-h}^R \lambda'(r) dr\\
  & \geq \int_{R-h}^R \left( f(r) - \lambda^2(r)\right) dr\\
  & \geq \frac{1}{2} h \mu^2,
\end{aligned}
$$
a contradiction.  This proves that $\lambda > \mu > 0$.\\

We first concentrate on the case $0 < a < 2$. Denote $\lambda_\pm = 1 \pm C e^{-a r}$. If $C>0$ is large enough ($C \geq \frac{J}{2-a}$), $\lambda_+$ satisfies the following inequalities:

$$
\left\{
\begin{aligned}
    \lambda_+' + \lambda_+^2 = 1 + (2-a) C e^{-ar} + C^2 e^{-2ar} > f(r),\\
    \lambda_+(0) > \lambda(0).
\end{aligned}
\right.
$$

From here, it is easy to prove that $\lambda < \lambda_+$ on $\bR^*_+$ as claimed (see \cite{Bahuaud} for instance). The reverse inequality $\lambda \geq \lambda_-$ cannot be proven so easily due to the $C^2 e^{-2ar}$ term. We first show that $\lambda(r) \to_{r \to \infty} 1$. We have proven that $\lambda \leq \lambda_+$, so $\limsup \lambda \leq 1$. Select $\mu < \nu < 1$, recalling here that $\mu$ is such that $\lambda > \mu$.  If $R$ is large enough, $f(r) > \nu^2$ for all $r \geq R$. In particular, $\lambda > \lambda_\nu$ where $\lambda_\nu$ is the solution of the Riccati equation:

$$
\left\{
\begin{aligned}
\lambda_\nu' + \lambda_\nu^2 &= \nu^2, \text{and } \nonumber \\
\lambda_\nu(R) & = \mu.
\end{aligned}
\right.
$$
A straightforward calculation shows that $\lambda_\nu \to \nu$. This proves that $\liminf \lambda \geq \nu$. As $\nu \in (\mu, 1)$ is arbitrary we conclude that $\lambda \to 1$. Select $C > \frac{J}{2-a}$ and we remark that for $R > 0$ large enough

$$\forall r > R,\quad\lambda_-'(r) + \lambda_-^2(r) < 1 - J e^{-a r} \leq f(r).$$

The proof will be complete if we can prove that, by selecting a larger $C$, $\lambda(R) > \lambda_-(R)$. We want that $C$ and $R$ satisfy

$$
\left\{
\begin{array}{l}
    1 - (2-a) C e^{-a r} + C^2 e^{-2ar} = \lambda'_- (r) + \lambda_-^2(r) < 1 - J e^{-a r} \qquad \forall r > R,\\
    1 - C e^{-a R} = \lambda_-(R) \leq \sigma,
\end{array}
\right.
$$
where $\sigma$ is a constant such that $\sigma < \lambda(R)$. The second condition will be fulfilled if $C~e^{-aR}=1-\sigma$ (note that $R$ increases with $C$). We now rewrite the first constraint as:

$$C^2 < \left[(2-a) C - J\right] e^{a r},$$
this inequality is satisfied for all $r \geq R$ as long as it is satisfied for $r=R$ which we now assume. The equality $C e^{-a R} = 1-\epsilon'$ implies that this inequality can be rewritten:
$$(1-\sigma) C < \left[(2-a) C - J\right]$$
which is true for large $C$ provided that $1-\sigma < 2-a$. As $\lambda \to_{r \to \infty} 1$, we are free to choose $\sigma$ as close to $1$ as we want and, in particular, we can assume that the previous inequality is satisfied. This proves the estimate when $0 < a < 2$.

We now come to the case $a=2$.  The proof above no longer works.  Instead we rewrite the Riccati equation  as
\[ (\lambda-1)' + 2 (\lambda-1) = f-1 - (\lambda-1)^2, \]
so that
\[ \left[e^{2r} (\lambda-1)\right]' = e^{2r} (f-1) - e^{2r} (\lambda-1)^2. \]

By assumption $\left|e^{2r} (f-1)\right| \leq J$, and the previous estimate applied to an arbitrary $1 < a < 2$ proves that $\left|e^{2r}(\lambda-1)^2\right| \leq C(a, J, \lambda(0))^2 e^{-(2a-2)r}$.  These estimates can be used to show that $\left[e^{2r} (\lambda-1)\right]'$ is integrable over $(0, \infty)$:

\begin{eqnarray*}
e^{2r} (\lambda(r)-1) - (\lambda(0)-1)
    & = & \int_0^r \left[e^{2s} (\lambda(s)-1)\right]'ds\\
    & = & \int_0^r\left[e^{2r} (f-1) - e^{2r} (\lambda-1)^2\right]\\
\left|e^{2r} (\lambda(r)-1)\right|
    & \leq & |\lambda(0)-1| + J r + \frac{C(a, J, \lambda(0))^2}{2a-2}\\
    & \leq & (r+1) C(2, J, \lambda(0)),
\end{eqnarray*}
for $C(2, J, \lambda(0))$ large enough.\\

The case $a>2$ can be treated similarly.
\end{proof}

The proof of the following theorem follows the same method as its analogue in \cite{Bahuaud} using Lemma \ref{lemma:scal-riccati-est} for the basic scalar estimate.

\begin{theorem}[Comparison theorem]\label{thm:comparison}  Given curvature assumptions \eqref{NSC} and \eqref{AH0}, let $(y^{\beta}, r)$ be Fermi coordinates for $Y$ on $W \times [0,\infty)$ for an open set $W \subset Y$.  Let $\Lambda, \lambda$ denote the maximum and minimum eigenvalues of the shape operator over $W$, and let $\Omega, \omega$ denote the maximum and minimum eigenvalue of the metric over $W$ (taken with respect to the background euclidean metric).  There exist positive constants $C, L_1$ and $L_2$ depending on these eigenvalues such that for $r$ sufficiently large we have
\flushleft\textbf{ Shape operator estimate: } 
\begin{eqnarray*}
 (1 - C e^{-ar})  \; \kronecker{\beta}{\alpha} \; \leq \; \sffud{\beta}{\alpha}(y, r) \; \leq \; (1 + C e^{-ar}) \; \kronecker{\beta}{\alpha} , \; \mbox{for} \; 0 < a < 2,
\end{eqnarray*}
\begin{eqnarray*}
 (1 - C re^{-2r})  \; \kronecker{\beta}{\alpha}  \; \leq \; \sffud{\beta}{\alpha}(y, r) \; \leq \; (1 + C re^{-2r}) \; \kronecker{\beta}{\alpha} , \; \mbox{for} \; a =2.
\end{eqnarray*}
\textbf{ Metric estimate: }
\begin{eqnarray*}
L_1 e^{2r} \; \delta_{\alpha \beta} \leq g_{\alpha \beta}(y, r) \leq L_2 \; e^{2r} \; \delta_{\alpha \beta}.
\end{eqnarray*}
\end{theorem}

% Analysis of the Riccati system
\section{Analysis of the Riccati system}
\label{section:riccati-analysis}

In this lengthy section we analyze systems of differential equations that arise from derivatives of the Riccati system \eqref{Riccati.System} and prove Theorems \ref{thm:A-part1} and \ref{thm:A-part2}.  We begin by deriving these systems of differential equations.  We then use the curvature hypothesis to estimate the various coefficients that appear in the system.  Proceeding in two iterations we compare the Riccati system to a model system with well understood asymptotics to obtain estimates for the compactified metric.  In the last part of this section we translate these decay estimates to H\"older estimates for the metric and prove Theorems \ref{thm:A-part1} and \ref{thm:A-part2}.

Fix an essential subset $K$, and set $\gbar = e^{-2r} g$.  In Fermi coordinates on $W \times [0,\infty)$, where $W$ is a sufficiently small open ball (see page \pageref{choosingW}), we may write
\[ g = dr^2 + g_{\alpha \beta}(y,r) dy^{\alpha} dy^{\beta}. \]

We set $\rho := e^{-r}$, and remind the reader of the convention given on page \pageref{page:greek} that $\rho$ does not count as a tangential or Greek variable.  In these `compactified Fermi coordinates', $(y^{\beta},\rho)$ over $W \times (0, 1]$, we now have
\[ g = \frac{d\rho^2}{\rho^2} + g_{\alpha \beta}(y, - \log \rho) dy^{\alpha} dy^{\beta}, \]
and consequently
\[ \gbar = d\rho^2 + \rho^2 g_{\alpha \beta}(y, - \log \rho) dy^{\alpha} dy^{\beta}. \]

We work out the first derivatives of $\gbar$:
\begin{equation} \label{eqn:gbar.normal.derivative}
\begin{aligned}
\partial_{\rho} \gbar_{\alpha \beta} &= 2 \rho g_{\alpha \beta} + \rho^2 \partial_r g_{\alpha \beta} \cdot \left(-\frac{1}{\rho}\right) \\ 
& = 2 \rho^{-1} (\delta^{\gamma}_{\alpha}-\sffud{\gamma}{\alpha}) \gbar_{\gamma \beta}, \\
\partial_{\mu} \gbar_{\alpha \beta} &= \rho^2 \partial_{\mu} g_{\alpha \beta}.
\end{aligned}
\end{equation}

We now take tangential derivatives of the Riccati system \eqref{Riccati.System}.  First observe
\begin{equation} \label{Tangential.System} 
\left\{ \begin{aligned}
 (\partial_{\mu} \sffud{\beta}{\alpha})' &= -(\partial_{\mu} \sffud{\beta}{\gamma}) \sffud{\gamma}{\alpha} - \sffud{\beta}{\gamma} (\partial_{\mu} \sffud{\gamma}{\alpha}) -\partial_{\mu} \riemdudd{0}{\beta}{\alpha}{0}, \\
 (\partial_{\mu} g_{\alpha \beta})' &= 2 (\partial_{\mu} \sffud{\gamma}{\alpha}) g_{\gamma \beta} + 2 \sffud{\gamma}{\alpha} (\partial_{\mu} g_{\gamma \beta}).
\end{aligned}
\right.
\end{equation}
In order to use the estimate for the shape operator, we rescale this system.  Set $W = e^{2r} \sff$, $\gbar = e^{-2r} g$.  The system becomes
\begin{equation} \label{Transformed.Tangential.System}
\left\{ \begin{aligned}
 (\partial_{\mu} \Wsffud{\beta}{\alpha})' & = - (\partial_{\mu} \Wsffud{\beta}{\gamma}) \sffud{\gamma}{\alpha} - \sffud{\beta}{\gamma} (\partial_{\mu} \Wsffud{\gamma}{\alpha}) + 2 \partial_{\mu} \Wsffud{\beta}{\alpha} - e^{2r}\partial_{\mu} \riemdudd{0}{\beta}{\alpha}{0} , \\
(\partial_{\mu} \gbar_{\alpha \beta})' & = 2 e^{-2r} (\partial_{\mu} \Wsffud{\gamma}{\alpha}) \gbar_{\gamma \beta} + 2 \sffud{\gamma}{\alpha} (\partial_{\mu} \gbar_{\gamma \beta}) - 2 (\partial_{\mu} \gbar_{\alpha \beta}).
\end{aligned}
\right.
\end{equation}

We also need equations for the second derivatives of the metric.

We have
\[ \partial^2_{\rho} \gbar_{\alpha \beta} = -2 \rho^{-2} (\delta^{\gamma}_{\alpha}-\sffud{\gamma}{\alpha}) \gbar_{\gamma \beta} + \rho^{-2} (\partial_r \sffud{\gamma}{\alpha}) \gbar_{\gamma \beta} + \rho^{-1} (\delta^{\gamma}_{\alpha}-\sffud{\gamma}{\alpha}) \partial_{\rho} \gbar_{\gamma \beta}.\]

Taking a tangential derivative of equation  \eqref{eqn:gbar.normal.derivative} yields
\begin{equation*}
\partial_{\mu} \partial_{\rho} \gbar_{\alpha \beta} = -2 \rho^{-1} (\partial_{\mu} \sffud{\gamma}{\alpha}) \gbar_{\gamma \beta} + 2 \rho^{-1} (\delta^{\gamma}_{\alpha}-\sffud{\gamma}{\alpha}) \partial_{\mu} \gbar_{\gamma \beta}.
\end{equation*}

The second tangential derivative of the Riccati system after the same rescaling as above is
\begin{equation} \label{Transformed.Second.Tangential.System}
\left\{ \begin{aligned}
 (\partial_{\nu} \partial_{\mu} \Wsffud{\beta}{\alpha})' & = - (\partial_{\nu}\partial_{\mu} \Wsffud{\beta}{\gamma}) \sffud{\gamma}{\alpha} - \sffud{\beta}{\gamma} (\partial_{\nu} \partial_{\mu} \Wsffud{\gamma}{\alpha}) + 2 \partial_{\nu} \partial_{\mu} \Wsffud{\beta}{\alpha} \\  
 &-e^{-2r}(\partial_{\mu} \Wsffud{\beta}{\gamma}) (\partial_{\nu} \Wsffud{\gamma}{\alpha}) - e^{-2r}(\partial_{\nu} \Wsffud{\beta}{\gamma}) (\partial_{\mu} \Wsffud{\gamma}{\alpha}) - e^{2r} \partial_{\nu} \partial_{\mu} \riemdudd{0}{\beta}{\alpha}{0}\\
(\partial_{\nu} \partial_{\mu} \gbar_{\alpha \beta})' & = 2 e^{-2r} (\partial_{\nu} \partial_{\mu} \Wsffud{\gamma}{\alpha}) \gbar_{\gamma \beta} + 2 \sffud{\gamma}{\alpha} (\partial_{\nu} \partial_{\mu} \gbar_{\gamma \beta}) - 2 (\partial_{\mu} \gbar_{\alpha \beta}) \\
&+ 2 e^{-2r}( \partial_{\mu} \Wsffud{\gamma}{\alpha})( \partial_{\nu}\gbar_{\gamma \beta}) + 2 e^{-2r} (\partial_{\nu} \Wsffud{\gamma}{\alpha}) (\partial_{\nu}\partial_{\mu} \gbar_{\gamma \beta}).
\end{aligned}
\right.
\end{equation}

The missing ingredient before we can begin an analysis in both systems \eqref{Transformed.Tangential.System} and \eqref{Transformed.Second.Tangential.System} is an estimate for the coordinate derivatives of curvature.  We can obtain estimates on these terms from the estimates \eqref{AH1} and \eqref{AH2} of the covariant derivatives of curvature and the addition of terms that couple derivatives of the metric into the equations.  The following lemmas provide these estimates.

\begin{lemma} \label{lemma:est.covg0gg0}
\begin{equation*}
 \partial_{\mu} \riemdudd{0}{\beta}{\alpha}{0} = - \Gamma^{\beta}_{\sigma \mu} (\riemdudd{0}{\sigma}{\alpha}{0} + \kronecker{\sigma}{\alpha}) + \Gamma^{\sigma}_{\alpha \mu} (\riemdudd{0}{\beta}{\sigma}{0} + \kronecker{\beta}{\sigma}) + O( e^{(1-a)r} )
\end{equation*}
\end{lemma}
The proof is straightforward.  See \cite{Bahuaud} for a proof.

\begin{lemma} \label{lemma:est.cov00gg0}
\begin{equation*}
 \partial_{0} \riemdudd{0}{\beta}{\alpha}{0} = - \Gamma^{\beta}_{\sigma \mu} (\riemdudd{0}{\sigma}{\alpha}{0} + \kronecker{\sigma}{\alpha}) + \Gamma^{\sigma}_{\alpha \mu} (\riemdudd{0}{\beta}{\sigma}{0} + \kronecker{\beta}{\sigma} ) + O( e^{-ar} )
\end{equation*}
\end{lemma}
The proof follows by straightforward computation.

The next two lemmas include initial estimates for the first tangential derivatives of the metric and shape operator as hypotheses.  Such estimates are available on the second iteration of the overall argument.  We assume $\gbar$ is $C^{0,1}$ which corresponds to our eventual application when the curvature decay $a > 1$.

\begin{lemma} \label{lemma:est.covgggg0} If $\gbar$ is $C^{0,1}$, then
\begin{align*} 
\partial_{\mu} \riemdudd{\lambda}{\beta}{\alpha}{0} &= O(e^{(2-a)r}).
\end{align*}
\end{lemma}
\begin{proof}
We begin with the standard formula relating covariant and coordinate derivatives:
\[ \nabla_{\mu} \riemdudd{\lambda}{\beta}{\alpha}{0}= \partial_{\mu}  \riemdudd{\lambda}{\beta}{\alpha}{0} - \Gamma^{s}_{\lambda \mu} \riemdudd{s}{\beta}{\alpha}{0} + \Gamma^{\beta}_{ s\mu} \riemdudd{\lambda}{s}{\alpha}{0} - \Gamma^{\sigma}_{\alpha \mu} \riemdudd{\lambda}{\beta}{\sigma}{0} - \Gamma^{\sigma}_{0 \mu} \riemdudd{\lambda}{\beta}{\alpha}{\sigma}.\]
Note the placement of the Greek index $\sigma$ over the Latin index $s$ in some of the contractions above are due to form of Christoffel symbols in Fermi coordinates.

We obtain
\begin{align} \label{eqn:locref1}
\partial_{\mu} \riemdudd{\lambda}{\beta}{\alpha}{0} &= \nabla_{\mu} \riemdudd{\lambda}{\beta}{\alpha}{0} + \Gamma^{s}_{\lambda \mu} \riemdudd{s}{\beta}{\alpha}{0} - \Gamma^{\beta}_{ s\mu} \riemdudd{\lambda}{s}{\alpha}{0} + \Gamma^{\sigma}_{\alpha \mu} \riemdudd{\lambda}{\beta}{\sigma}{0} + \Gamma^{\sigma}_{0 \mu} \riemdudd{\lambda}{\beta}{\alpha}{\sigma} \nonumber \\
&= \nabla_{\mu} \riemdudd{\lambda}{\beta}{\alpha}{0} + \Gamma^{\sigma}_{\lambda \mu} \riemdudd{\sigma}{\beta}{\alpha}{0} - \Gamma^{\beta}_{\sigma\mu} \riemdudd{\lambda}{\sigma}{\alpha}{0} + \Gamma^{\sigma}_{\alpha \mu} \riemdudd{\lambda}{\beta}{\sigma}{0} \\
&+ \Gamma^{\sigma}_{0 \mu} \riemdudd{\lambda}{\beta}{\alpha}{\sigma} + \Gamma^{0}_{\lambda \mu} \riemdudd{0}{\beta}{\alpha}{0}- \Gamma^{\beta}_{0\mu} \riemdudd{\lambda}{0}{\alpha}{0} \nonumber
\end{align}
 
We consider the last three terms in \eqref{eqn:locref1}.  We use estimate \eqref{AH0} and the fact that in Fermi coordinates $\Gamma^{0}_{\lambda \mu} = -\sffdd{\lambda}{\mu}$ and $\Gamma^{\sigma}_{0 \mu} = \sffud{\sigma}{\mu}$, to obtain  
\begin{align*}
\Gamma^{\sigma}_{0 \mu} \riemdudd{\lambda}{\beta}{\alpha}{\sigma} &+ \Gamma^{0}_{\lambda \mu} \riemdudd{0}{\beta}{\alpha}{0}- \Gamma^{\beta}_{0\mu} \riemdudd{\lambda}{0}{\alpha}{0} \\
& = \sffud{\sigma}{\mu} \riemdudd{\lambda}{\beta}{\alpha}{\sigma} -\sffdd{\lambda}{\mu} \riemdudd{0}{\beta}{\alpha}{0}
 + \sffud{\beta}{\mu} \riemuddd{0}{\lambda}{\alpha}{0} \\
&= \sffud{\sigma}{\mu} ( g_{\lambda \sigma} \kronecker{\beta}{\alpha} - g_{\lambda \alpha} \kronecker{\beta}{\sigma}) -\sffdd{\lambda}{\mu} \kronecker{\beta}{\alpha} + \sffud{\beta}{\mu} g_{\lambda \alpha} +  O(e^{-ar}) \\
&= O(e^{-ar})
\end{align*}

Note that for all Greek indices, $\Gamma^{\sigma}_{\alpha \mu} = \Gambar^{\sigma}_{\alpha \mu} = O(1)$ by assumption that $\gbar$ is Lipschitz.  Consequently, the three terms in \eqref{eqn:locref1} like $\Gamma^{\sigma}_{\alpha \mu} \riemdudd{\lambda}{\beta}{\sigma}{0} = O(e^{(1-a)r})$.

Collecting these estimates we find,
\begin{align*}
\partial_{\mu} \riemdudd{\lambda}{\beta}{\alpha}{0} &= \nabla_{\mu} \riemdudd{\lambda}{\beta}{\alpha}{0} + O(e^{(1-a)r}) + O(e^{-ar}) \nonumber\\
&= O(e^{(2-a)r})+ O(e^{(1-a)r}) + O(e^{-ar}) \nonumber \\
&= O(e^{(2-a)r}).
\end{align*}
\end{proof} 

\begin{lemma} \label{lemma:est.sec.cov.der} If $\gbar$ is $C^{0,1}$ and the tangential derivatives of the shape operator satisfy $\partial_{\mu} \sffud{\beta}{\alpha} = O(1)$, then
\begin{align*}
\partial_{\nu} \partial_{\mu} \riemdudd{0}{\beta}{\alpha}{0} &= (\partial \gbar * (\riem+\delta))_{\phantom{\beta}\alpha}^{\beta} + O( e^{(2-a)r} ),
\end{align*}
where $\partial \gbar * (\riem+\delta)$ denotes terms that are bounded coefficients multiplied by contractions of first tangential derivatives of $\gbar$ and $\riem+\delta$.
\end{lemma}
\begin{proof}
We again begin with a standard formula relating covariant and coordinate derivatives.
\begin{align} \label{eqn:locref2}
\nabla_{\nu} \nabla_{\mu} \riemdudd{0}{\beta}{\alpha}{0} &= \partial_{\nu} (\nabla_{\mu}  \riemdudd{0}{\beta}{\alpha}{0}) - \Gamma^{s}_{\nu \mu} (\nabla_{s}  \riemdudd{0}{\beta}{\alpha}{0}) - \Gamma^{\sigma}_{\nu 0}( \nabla_{\mu}  \riemdudd{\sigma}{\beta}{\alpha}{0} ) \nonumber \\
& + \Gamma^{\beta}_{\nu \sigma}( \nabla_{\mu}  \riemdudd{0}{\sigma}{\alpha}{0} ) - \Gamma^{\sigma}_{\nu \alpha} (\nabla_{\mu}  \riemdudd{0}{\beta}{\sigma}{0} ) - \Gamma^{\sigma}_{\nu 0} (\nabla_{\mu}  \riemdudd{0}{\beta}{\alpha}{\sigma}) \\
&= \partial_{\nu} (\nabla_{\mu}  \riemdudd{0}{\beta}{\alpha}{0}) - \Gamma^{0}_{\nu \mu} (\nabla_{0}  \riemdudd{0}{\beta}{\alpha}{0}) - \Gamma^{\sigma}_{\nu \mu} (\nabla_{\sigma}  \riemdudd{0}{\beta}{\alpha}{0}) - \Gamma^{\sigma}_{\nu 0}( \nabla_{\mu}  \riemdudd{\sigma}{\beta}{\alpha}{0} ) \nonumber \\
& + \Gamma^{\beta}_{\nu \sigma}( \nabla_{\mu}  \riemdudd{0}{\sigma}{\alpha}{0} ) - \Gamma^{\sigma}_{\nu \alpha} (\nabla_{\mu}  \riemdudd{0}{\beta}{\sigma}{0} ) - \Gamma^{\sigma}_{\nu 0} (\nabla_{\mu}  \riemdudd{0}{\beta}{\alpha}{\sigma}) \nonumber.
\end{align}

We first consider estimates for the contractions $\Gamma * \nabla \riem$ that appear above.  The behaviour of the contractions come in three families based on the placement of the indices.  We estimate a typical member of these families using the estimates on page \pageref{pg:curvests} and the fact that $\gbar$ is Lipschitz as follows
\[ \Gamma^{0}_{\nu \mu} (\nabla_{0} \riemdudd{0}{\beta}{\alpha}{0}) = \sffdd{\mu}{\nu} O(e^{-ar}) = O(e^{(2-a)r}).\]
\[ \Gamma^{\sigma}_{\nu \mu} (\nabla_{\sigma} \riemdudd{0}{\beta}{\alpha}{0} )=  \Gambar^{\sigma}_{\nu \mu} (\nabla_{\sigma} \riemdudd{0}{\beta}{\alpha}{0} )= O(e^{(2-2a)r}).\]
\[ \Gamma^{\sigma}_{\nu 0}( \nabla_{\mu} \riemdudd{\sigma}{\beta}{\alpha}{0}) = O(e^{(2-a)r}).\]

Consequently all of the contractions $\Gamma * \nabla \riem$ that appear are at worst $O(e^{(2-a)r})$.

We now expand the first term of \eqref{eqn:locref2} above.
\begin{align*}
\partial_{\nu} (\nabla_{\mu} \riemdudd{0}{\beta}{\alpha}{0}) &= \partial_{\nu} \left( \partial_{\mu} \riemdudd{0}{\beta}{\alpha}{0} - \Gamma^{\sigma}_{\mu 0} \riemdudd{\sigma}{\beta}{\alpha}{0} + \Gamma^{\beta}_{\mu \sigma} \riemdudd{0}{\sigma}{\alpha}{0} - \Gamma^{\sigma}_{\mu \alpha} \riemdudd{0}{\beta}{\sigma}{0}  - \Gamma^{\sigma}_{\mu 0} \riemdudd{0}{\beta}{\alpha}{\sigma} \right) \\
&= \partial_{\nu} \partial_{\mu} \riemdudd{0}{\beta}{\alpha}{0}
- \partial_{\nu} \Gamma^{\sigma}_{\mu 0} \riemdudd{\sigma}{\beta}{\alpha}{0} 
- \Gamma^{\sigma}_{\mu 0} \partial_{\nu} \riemdudd{\sigma}{\beta}{\alpha}{0} 
+ \partial_{\nu}\Gamma^{\beta}_{\mu \sigma} \riemdudd{0}{\sigma}{\alpha}{0}
+ \Gamma^{\beta}_{\mu \sigma} \partial_{\nu} \riemdudd{0}{\sigma}{\alpha}{0} \\
&- \partial_{\nu}\Gamma^{\sigma}_{\mu \alpha} \riemdudd{0}{\beta}{\sigma}{0}
- \Gamma^{\sigma}_{\mu \alpha} \partial_{\nu}\riemdudd{0}{\beta}{\sigma}{0}
- \partial_{\nu}\Gamma^{\sigma}_{\mu 0} \riemdudd{0}{\beta}{\alpha}{\sigma}
- \Gamma^{\sigma}_{\mu 0} \partial_{\nu}\riemdudd{0}{\beta}{\alpha}{\sigma}.
\end{align*}

We again examine representative behaviour of the terms.  First we have
\[ \Gamma^{\sigma}_{\mu 0} \partial_{\nu} \riemdudd{\sigma}{\beta}{\alpha}{0} =  O(e^{(2-a)r}),\]
by Lemma \ref{lemma:est.covgggg0} and the fact $\Gamma^{\sigma}_{\mu 0} = \sffud{\sigma}{\mu}$.  Next, Lemma \ref{lemma:est.covg0gg0} and the fact that $\gbar$ is Lipschitz yield
\[ \Gamma^{\sigma}_{\mu \alpha} \partial_{\nu} \riemdudd{0}{\beta}{\sigma}{0} = O(e^{(1-a)r}).\]
Finally, as $\gbar$ is Lipschitz and the tangential derivatives of $\sffud{\sigma}{\mu}$ are bounded,
\[ \partial_{\nu} \Gamma^{\sigma}_{\mu 0} \riemdudd{\sigma}{\beta}{\alpha}{0} = \partial_{\nu} \sffud{\sigma}{\mu} \riemdudd{\sigma}{\beta}{\alpha}{0}= O(e^{(1-a)r}).\]
Now 
\begin{align*}
-\partial_{\nu} \Gamma^{\sigma}_{\mu \alpha} \riemdudd{0}{\beta}{\sigma}{0} &+ \partial_{\nu}\Gamma^{\beta}_{\mu \sigma} \riemdudd{0}{\sigma}{\alpha}{0} \\
&=-\partial_{\nu} \Gamma^{\sigma}_{\mu \alpha} (\riemdudd{0}{\beta}{\sigma}{0}+\kronecker{\beta}{\sigma}) + \partial_{\nu}\Gamma^{\beta}_{\mu \sigma} (\riemdudd{0}{\sigma}{\alpha}{0} + \kronecker{\sigma}{\alpha}) \\
&=-\partial_{\nu} \Gambar^{\sigma}_{\mu \alpha} (\riemdudd{0}{\beta}{\sigma}{0}+\kronecker{\beta}{\sigma}) + \partial_{\nu}\Gambar^{\beta}_{\mu \sigma} (\riemdudd{0}{\sigma}{\alpha}{0} + \kronecker{\sigma}{\alpha}).
\end{align*}
We now apply the derivative to the Christoffel symbols.  The product rule yields sums of terms that are contractions of second tangential partial derivatives of $\gbar$ contracted with $\riem+\delta$ and remainder terms involving only first tangential derivatives of $\gbar$ which we may estimate:
\begin{align*}
-\partial_{\nu} \Gamma^{\sigma}_{\mu \alpha} \riemdudd{0}{\beta}{\sigma}{0} &+ \partial_{\nu}\Gamma^{\beta}_{\mu \sigma} \riemdudd{0}{\sigma}{\alpha}{0} \\
&=-\partial_{\nu} \Gambar^{\sigma}_{\mu \alpha} (\riemdudd{0}{\beta}{\sigma}{0}+\kronecker{\beta}{\sigma}) + \partial_{\nu}\Gambar^{\beta}_{\mu \sigma} (\riemdudd{0}{\sigma}{\alpha}{0} + \kronecker{\sigma}{\alpha}) \\
&= (\partial \gbar * (\riem + \delta))^{\beta}_{\phantom{\beta}\alpha} + O(e^{-ar})
\end{align*}

Collecting everything above, using the worst case estimate yields
\begin{align*}
\partial_{\nu} \partial_{\mu} \riemdudd{0}{\beta}{\alpha}{0} &= \nabla_{\nu} \nabla_{\mu} \riemdudd{0}{\beta}{\alpha}{0} + (\partial \gbar *  (\riem + \delta))^{\beta}_{\phantom{\beta}\alpha} + O( e^{(2-a)r} )
\end{align*}
\end{proof}

We now proceed with the analysis of systems \eqref{Transformed.Tangential.System} and \eqref{Transformed.Second.Tangential.System}.  In what follows we regard these systems as systems of ODEs in new dependent variables.   For example, we regard the components $\partial_{\mu} \Wsffud{\beta}{\alpha}$ and $\partial_{\mu} \gbar_{\alpha \beta}$ in \eqref{Transformed.Tangential.System} as vectors in $\mathbb{R}^{n^3}$ which we denote $\partial W$ and $\partial \gbar$.  The above system may be compactly written as:
\begin{equation} \label{eqn:tansys3} 
\left\{ \begin{aligned}
(\partial W)' &= A \partial W + B \partial \gbar + H_1, \\
(\partial \gbar)' &= C \partial W + D \partial \gbar + H_2.
\end{aligned}
\right.
\end{equation}
where $A,B,C,D, H_1$ and $H_2$ are $(n^3 \times n^3)$-matrices.  We will not need the explicit form of these matrices in what follows; we only need estimates on the size of the matrix entries.  An entirely similar discussion holds for $\eqref{Transformed.Second.Tangential.System}$.

We now state our main comparison result for systems of this form.  For a proof see \cite[Appendix 3.1]{Bahuaud}.

\begin{theorem} \label{thm:ODEcomp} Suppose that $a, b, c, d, e, f$ are smooth functions on $[t_0, t_1]$ (respectively $[t_0, \infty)$) with $a, b, c, d$ positive.  Suppose that $x$ and $y$ are nonnegative continuous functions that are smooth where they are nonzero and satisfy the differential inequalities
\begin{align*}
x' & \leq a x + by + e, \\
y' & \leq c x + dy + f.
\end{align*}
Suppose in addition that $u$ and $v$ are positive smooth solutions of the corresponding system of differential equations:
\begin{align*}
u' & = a u + bv + e, \\
v' & = c u + dv + f.
\end{align*}
If $x(t_0) < u(t_0)$ and $y(t_0) < v(t_0)$ then $x < u$ and $y < v$ on  $[t_0, t_1]$ (respectively $[t_0, \infty)$).
\end{theorem}

We are now ready to obtain our estimates for the first derivatives of the compactified metric and shape operator.

\begin{prop} \label{prop:firstiteration} Given curvature assumptions \eqref{AH0}, \eqref{AH1}, the first derivatives of the shape operator and compactified metric satisfy:
\[ \partial_{\mu} \sffud{\beta}{\alpha} = O(e^{(1-a)r}), \; \; 0 < a \leq 2\]
\begin{equation*} 
\partial_{r} \sffud{\beta}{\alpha} =
\left\{ \begin{aligned}
&O(e^{-ar}),& 0 < a < 2 \\
&O(r e^{-2r}),  & a=2
\end{aligned}
\right.
\end{equation*}
\begin{equation*} 
\partial_{\mu} \gbar_{\alpha \beta} = 
\left\{ \begin{aligned}
&O(e^{(1-a)r}),& 0 < a < 1 \\
&O(r),& a=1 \\
&O(1),& a > 1
\end{aligned}
\right.
\end{equation*}
\[ \partial_{\rho} \gbar_{\alpha \beta} = O(1). \]
\caution we provide estimates for $\sffud{\beta}{\alpha}$ in uncompactified coordinates $(r,y^{\alpha})$ but estimates for $\gbar$ in compactified coordinates $(\rho,y^{\alpha})$! 
\end{prop}
\begin{proof}
Inserting the estimate for the coordinate derivative of curvature from Lemma \ref{lemma:est.covg0gg0} into 
\eqref{Transformed.Tangential.System}, we obtain
\begin{equation*}
\left\{ \begin{aligned}
 (\partial_{\mu} \Wsffud{\beta}{\alpha})' & = - (\partial_{\mu} \Wsffud{\beta}{\gamma}) \sffud{\gamma}{\alpha} - \sffud{\beta}{\gamma} (\partial_{\mu} \Wsffud{\gamma}{\alpha}) + 2 \partial_{\mu} \Wsffud{\beta}{\alpha} \\
 &+ e^{2r}
\Gamma^{\sigma}_{\alpha \mu} (\riemdudd{0}{\beta}{\sigma}{0} + \kronecker{\beta}{\sigma}) - e^{2r} \Gamma^{\beta}_{\sigma \mu} (\riemdudd{0}{\sigma}{\alpha}{0} + \kronecker{\sigma}{\alpha}) + O( e^{(3-a)r} ), \\ 
(\partial_{\mu} \gbar_{\alpha \beta})' & = 2 e^{-2r} (\partial_{\mu} \Wsffud{\gamma}{\alpha}) \gbar_{\gamma \beta} + 2 \sffud{\gamma}{\alpha} (\partial_{\mu} \gbar_{\gamma \beta}) - 2 (\partial_{\mu} \gbar_{\alpha \beta}).
\end{aligned}
\right.
\end{equation*}
This system is of the form \eqref{eqn:tansys3} with coefficient estimates $A = O(e^{-ar})$ when $0 < a < 2$ and $A = O(re^{-2r})$ when $a=2$.  Also we have $B = O(e^{(2-a)r})$, $C = O(e^{-2r})$, $D = O(e^{-ar})$, $H_1 = O(e^{(3-a)r})$, $H_2 = 0$.  We compare this system to the model system
\begin{equation*} 
\left\{ \begin{aligned}
u' & =  ce^{-ar} u +  ce^{(2-a)r} v +  ce^{(3-a)r}, \\
v' & =  ce^{-2r} u + ce^{-ar} v,
\end{aligned}
\right.
\end{equation*}
for some constant $c > 0$, when $0< a < 2$, and we compare to the model system 
\begin{equation*} 
\left\{ \begin{aligned}
u' & =  cre^{-2r} u +  c v +  ce^{r}, \\
v' & =  ce^{-2r} u + cre^{-2r} v,
\end{aligned}
\right.
\end{equation*}
when $a=2$.  Note that solutions to this comparison system with positive initial conditions remain positive.

Set $x(r) = |\partial W|$ and $y(r) = |\partial \gbar|$.  These functions are continuous and smooth where they are nonzero.  The Cauchy-Schwarz inequality implies that when $x$ and $y$ are nonzero, $x' \leq |(\partial W)'|$ and $y' \leq |(\partial \gbar)'|$.  Applying this and the coefficient estimates to our system implies
\begin{equation*} 
\left\{ \begin{aligned}
x' & \leq  ce^{-ar} x +  ce^{(2-a)r} y +  ce^{(3-a)r}, \\
y' & \leq  ce^{-2r} x + ce^{-ar} y,
\end{aligned}
\right.
\end{equation*}

By Theorem \ref{thm:ODEcomp} and the analysis of the appendix, we find that the solutions satisfy estimates 
\begin{equation*} 
\partial W = 
\left\{ \begin{aligned}
O(e^{(3-a)r}),& 0 < a \leq 2 
\end{aligned}
\right.
\end{equation*}
and 
\begin{equation*} 
\partial \gbar = 
\left\{ \begin{aligned}
&O(e^{(1-a)r}),& 0 < a < 1 \\
&O(r),& a=1 \\
&O(1),& a > 1
\end{aligned}
\right.
\end{equation*}

For the estimates for $\rho$ derivatives, observe that $\partial_{\rho} \gbar_{\alpha \beta} = 2 \rho^{-1} (\kronecker{\gamma}{\alpha}-\sffud{\gamma}{\alpha}) \gbar_{\gamma \beta} = O(1)$ by the shape operator estimate from Theorem \ref{thm:comparison}.

We may estimate $\partial_r \sffud{\gamma}{\alpha}$ using the Riccati equation and the estimates for the shape operator with \eqref{AH0}:
\begin{equation*} 
 \partial_r \sffud{\beta}{\alpha} = - \sffud{\beta}{\mu} \sffud{\mu}{\alpha} - (\riem_{\partial r})^{\beta}_{\phantom{\beta}\alpha} =
\left\{ \begin{aligned}
&O(e^{-ar}),& 0 < a < 2 \\
&O(r e^{-2r}) = O( \rho^2\log \rho), & a=2
\end{aligned}
\right.
\end{equation*}
\end{proof}

We now perform the second iteration of the argument to estimate second derivatives of the metric
\begin{prop} \label{prop:seconditeration} Given curvature assumptions \eqref{AH0}, \eqref{AH1} and \eqref{AH2} for $a > 1$, the second derivatives of the shape operator and compactified metric satisfy:
\begin{equation*} 
\partial^2_{\mu \nu} \sffud{\beta}{\alpha} =
\left\{ \begin{aligned}
&O(e^{(2-a)r}),& 0 < a < 2 \\
&O( r ), &a=2
\end{aligned}
\right.
\end{equation*}
\begin{equation*} 
\partial^2_r \sffud{\beta}{\alpha} = 
\left\{ \begin{aligned}
&O(e^{-ar}),& 0 < a < 2 \\
&O(r e^{-2r}), & a=2
\end{aligned}
\right.
\end{equation*}
\[ \partial_r \partial_{\mu} \sffud{\beta}{\alpha} = O(e^{(1-a)r}) \]

\begin{equation*} 
\partial^2_{\mu \nu} \gbar_{\alpha \beta},  \partial^2_{\rho \rho} \gbar_{\alpha \beta}= 
\left\{ \begin{aligned}
&O(e^{(2-a)r}) = O(\rho^{a-2}),& 0 < a < 2 \\
&O(r) = O( \log \rho), & a=2
\end{aligned}
\right.
\end{equation*}
\[ \partial^2_{\mu \rho} \gbar_{\alpha \beta} = O(1).\]
\end{prop}
\begin{proof}
Note that as $a > 1$, Proposition \ref{prop:firstiteration} provides estimates for the first derivatives of the shape operator and compactified metric.

For the second tangential derivatives, we now insert the estimates from Lemma \ref{lemma:est.covgggg0} and \ref{lemma:est.sec.cov.der} to obtain the system
\begin{equation*}  
\left\{ \begin{aligned}
 (\partial_{\nu} \partial_{\mu} \Wsffud{\beta}{\alpha})' & = - (\partial_{\nu}\partial_{\mu} \Wsffud{\beta}{\gamma}) \sffud{\gamma}{\alpha} - \sffud{\beta}{\gamma} (\partial_{\nu} \partial_{\mu} \Wsffud{\gamma}{\alpha}) + 2 \partial_{\nu} \partial_{\mu} \Wsffud{\beta}{\alpha} \\  
 &-e^{-2r}(\partial_{\mu} \Wsffud{\beta}{\gamma}) (\partial_{\nu} \Wsffud{\gamma}{\alpha}) - e^{-2r}(\partial_{\nu} \Wsffud{\beta}{\gamma}) (\partial_{\mu} \Wsffud{\gamma}{\alpha}) \\
 &- e^{2r}( (\partial \gbar * (\riem+\delta))_{\phantom{\beta}\alpha}^{\beta} + O( e^{(2-a)r} ) ) \\
(\partial_{\nu} \partial_{\mu} \gbar_{\alpha \beta})' & = 2 e^{-2r} (\partial_{\nu} \partial_{\mu} \Wsffud{\gamma}{\alpha}) \gbar_{\gamma \beta} + 2 \sffud{\gamma}{\alpha} (\partial_{\nu} \partial_{\mu} \gbar_{\gamma \beta}) - 2 (\partial_{\mu} \gbar_{\alpha \beta}) \\
&+ 2 e^{-2r}( \partial_{\mu} \Wsffud{\gamma}{\alpha})( \partial_{\nu}\gbar_{\gamma \beta}) + 2 e^{-2r} (\partial_{\nu} \Wsffud{\gamma}{\alpha}) (\partial_{\nu}\partial_{\mu} \gbar_{\gamma \beta}).
\end{aligned}
\right.
\end{equation*}
This system is again of the form \eqref{eqn:tansys3} with coefficient estimates $A = O(e^{-ar})$ $(0 < a < 2)$ and $A=O(re^{-2r})$ $(a=2)$, $B = O(e^{(2-a)r})$, $C = O(e^{-2r})$, $D = O(e^{-ar})$, $H_1 = O(e^{(4-a)r})$, $H_2 = O(e^{(1-a)r})$.  We again compare this system to a model system.  By Theorem \ref{thm:ODEcomp} and the analysis of the appendix, we find that the solutions satisfy estimates 
\begin{equation*} 
|\partial^2 W| =
\left\{ \begin{aligned}
&O(e^{(4-a)r}),& 0 < a < 2 \\
&O( r e^{2r} ), &a=2
\end{aligned}
\right.
\end{equation*}
and 
\begin{equation*} 
|\partial^2 \gbar| =
\left\{ \begin{aligned}
&O(e^{(2-a)r}),& 0 < a < 2 \\
&O(r),& a=2
\end{aligned}
\right.
\end{equation*}

For the second $\rho$ derivatives we find
\[ \partial^2_{\rho} \gbar_{\alpha \beta} = -2 \rho^{-2} (\delta^{\gamma}_{\alpha}-\sffud{\gamma}{\alpha}) \gbar_{\gamma \beta} + \rho^{-2} (\partial_r \sffud{\gamma}{\alpha}) \gbar_{\gamma \beta} + \rho^{-1} (\delta^{\gamma}_{\alpha}-\sffud{\gamma}{\alpha}) \partial_{\rho} \gbar_{\gamma \beta}.\]
Consequently we find
\begin{equation*} 
\partial^2_{\rho} \gbar_{\alpha \beta} = 
\left\{ \begin{aligned}
&O(e^{(2-a)r}) = O(\rho^{a-2}),& 0 < a < 2 \\
&O(r) = O( \log \rho), & a=2
\end{aligned}
\right.
\end{equation*}

Taking a tangential derivative of equation \eqref{eqn:gbar.normal.derivative} yields
\begin{equation*}
\partial_{\mu} \partial_{\rho} \gbar_{\alpha \beta} = -2 \rho^{-1} (\partial_{\mu} \sffud{\gamma}{\alpha}) \gbar_{\gamma \beta} + 2 \rho^{-1} (\delta^{\gamma}_{\alpha}-\sffud{\gamma}{\alpha}) \partial_{\mu} \gbar_{\gamma \beta}.
\end{equation*}

Given the estimates from Proposition \ref{prop:firstiteration}, we find
\begin{equation*} 
\partial_{\mu} \partial_{\rho} \gbar_{\alpha \beta}= 
\left\{ \begin{aligned}
&O(e^{(2-a)r}) = O(\rho^{a-2}),& 0 < a < 2 \\
&O(1), & a=2
\end{aligned}
\right.
\end{equation*}

Finally we work out the remaining estimates for second derivatives of $\sffud{\beta}{\alpha}$.  Differentiating the Riccati equation yields 
\begin{equation*} 
 \partial^2_r \sffud{\beta}{\alpha} = - (\partial_r \sffud{\beta}{\mu}) \sffud{\mu}{\alpha} - \sffud{\beta}{\mu} (\partial_r \sffud{\mu}{\alpha}) - \partial_r (\riem_{\partial r})^{\beta}_{\phantom{\beta}\alpha}.
\end{equation*}
Note that from Lemma \ref{lemma:est.cov00gg0},
\[ \partial_r (\riem_{\partial r})^{\beta}_{\phantom{\beta}\alpha} = -\Gamma^{\sigma}_{\alpha \mu} (\riemdudd{0}{\beta}{\sigma}{0} + \kronecker{\beta}{\sigma} )+ \Gamma^{\beta}_{\sigma \mu} (\riemdudd{0}{\sigma}{\alpha}{0} + \kronecker{\beta}{\sigma}) + O( e^{-ar} ) = O(e^{-ar}).\]
Combined with the estimate for $\partial_r \sffud{\beta}{\mu}$ from Proposition \ref{prop:firstiteration}, we have
\begin{equation*} 
\partial^2_r \sffud{\beta}{\alpha} = 
\left\{ \begin{aligned}
&O(e^{-ar}),& 0 < a < 2 \\
&O(r e^{-2r}), & a=2
\end{aligned}
\right.
\end{equation*}
Finally, the first equation of \eqref{Tangential.System} combined with the estimates Proposition \ref{prop:firstiteration} allows us to estimate the mixed derivatives of $\sffud{\beta}{\alpha}$.
\begin{equation*} 
 \partial_r \partial_{\mu} \sffud{\beta}{\alpha} = -(\partial_{\mu} \sffud{\beta}{\gamma}) \sffud{\gamma}{\alpha} - \sffud{\beta}{\gamma} (\partial_{\mu} \sffud{\gamma}{\alpha}) -\partial_{\mu} \riemdudd{0}{\beta}{\alpha}{0} = O(e^{(1-a)r}), \; 0 < a \leq 2
\end{equation*}
\end{proof}

We now present a Lemma that allows us to convert decay estimates for functions into H\"older estimates.

\begin{lemma} \label{lemma:holderreg}
Suppose that $F$ is a function in compactified Fermi coordinates $W \times [0, \epsilon)$ that is smooth for $\rho > 0$.
\begin{enumerate}
\item If $0 < a < 1$ and all coordinate derivatives of $F$ satisfy
\[ \partial F(p, \rho) = O(\rho^{a-1}),\]
Then $F \in C^{0,a}(W \times [0, \epsilon))$.
\item If all coordinate derivatives of $F$ satisfy
\[ \partial F(p, \rho) = O(\log \rho),\]
Then $F \in C^{0,b}(W \times [0, \epsilon))$, for every $0 < b < 1$.
\item If $1 < a < 2$ and all second coordinate derivatives of $F$ satisfy
\[ \partial^2 F(p, \rho) = O(\rho^{a-2}),\]
Then $F \in C^{1,a-1}(W \times [0, \epsilon))$.
\item If all second coordinate derivatives of $F$ satisfy
\[ \partial^2 F(p, \rho) = O(\log \rho),\]
Then $F \in C^{1,b}(W \times [0, \epsilon))$, for every $0 < b < 1$.
\end{enumerate}
\end{lemma}
\begin{proof}
We will only give the proof for the first two cases, the others being straightforward generalizations.  For the first case take a truncated cylinder $W \times (0, \epsilon)$ on which 
\[ |\partial_\mu F(p, \rho)| \leq C \rho^{a-1}, \; \mbox{and} \]
\[ |\partial_\rho F(p, \rho)| \leq C \rho^{a-1} \]
for some constant $C > 0$ independent of $\rho \in (0, \epsilon)$ and $p \in W$.  We first remark that it is sufficient to prove only the following ``tangential" H\"older continuity:
$$
\frac{\left|F(p, \rho) - F(q, \rho)\right|}{\left|p-q\right|^a} \leq\tilde{C}\qquad\forall p, q \in W, p \neq q ~\text{and}~ \rho \in (0, \epsilon)
$$
for some constant $\tilde{C}$ independent of $p, q, \rho$. Indeed if $\rho, \rho' \in (0, \epsilon), \rho \neq \rho'$ :
\begin{eqnarray*}
\frac{\left|F(p, \rho) - F(q, \rho')\right|}{\left(\left|p-q\right|^2 + \left|\rho-\rho'\right|^2\right)^\frac{a}{2}}
&\leq& \frac{\left|F(p,\rho)-F(q,\rho)\right|}{\left(\left|p-q\right|^2+\left|\rho-\rho'\right|^2\right)^\frac{a}{2}}
 		 +\frac{\left|F(q,\rho)-F(q,\rho')\right|}{\left(\left|p-q\right|^2+\left|\rho-\rho'\right|^2\right)^\frac{a}{2}}\\
&\leq& \frac{\left|F(p,\rho)-F(q,\rho)\right|}{\left(\left|p-q\right|\right)^a}
		 + \frac{\left|F(q,\rho)-F(q,\rho')\right|}{\left|\rho-\rho'\right|^a}\\
&\leq& \tilde{C} + \frac{\left|F(q,\rho)-F(q,\rho')\right|}{\left|\rho-\rho'\right|^a}
\end{eqnarray*}

The second term can be easily estimated (assume $0 < \rho' < \rho$):

\begin{eqnarray*}
F(q,\rho)-F(q,\rho') & = & \int_{\rho'}^{\rho} \partial_\rho F(q, \sigma) d\sigma\\
\left|F(q,\rho)-F(q,\rho')\right| &\leq& \int_{\rho'}^{\rho} \left|\partial_\rho F(q, \sigma)\right| d\sigma\\
	& \leq & C \int_{\rho'}^{\rho} \sigma^{a-1} d\sigma\\
	& \leq & C (\rho - \rho') \int_0^1 \left( \rho x + (1-x) \rho' \right)^{a-1} dx \qquad \left(\sigma = \rho x + (1-x) \rho'\right)\\
	& \leq & C \left(\rho-\rho'\right)^a \int_0^1 x^{a-1} dx
								\qquad \text{because}~\left(\rho'+x(\rho-\rho')\right)^{a-1} \leq \left(x(\rho-\rho')\right)^{a-1}\\
	& \leq & \frac{C}{a} \left(\rho-\rho'\right)^a
\end{eqnarray*}

So we need only estimate the tangential H\"older inequality.  Let $p$ and $q$ be two points in $W$ and denote
$d = \left|p-q\right|$ the (euclidean) distance between $p$ and $q$ in the chart and assume $d < 1$. We distinguish two cases.  First assume that $\rho \geq d$:

\begin{eqnarray*}
\left|F(p, \rho) - F(q, \rho)\right| &  =  & \left|\int_0^1 \left(p^\mu - q^\mu\right) \partial_\mu F\right|\\
		& \leq& \left|p^\mu - q^\mu\right| \sup_{p' \in \Omega_0} \left|\partial_{\mu} F(p', \rho)\right|\\
		& \leq& C~d~\rho^{a-1}\\
		& \leq& C \left(\frac{d}{\rho}\right)^{1-a} d^a\\
		& \leq& C~d^a
\end{eqnarray*}

Assuming now that $\rho \leq d$, we ``lift" the inequality to $\rho = d$:
\[ \begin{array}{rcccccc}
\left|F(p, \rho) - F(q, \rho)\right| & \leq &
			\left|F(p, \rho) - F(p,    d)\right|
	&+& \left|F(p,    d) - F(q,    d)\right|
	&+& \left|F(q,    d) - F(q, \rho)\right|\\
& \leq & C \int_\rho^d \sigma^{a-1} d \sigma & + & C d^a & + & C \int_\rho^d \sigma^{a-1} d \sigma\\
& \leq & \multicolumn{5}{l}{C d^a + 2 \frac{C}{a} \left(d^a - \rho^a \right)}\\
& \leq & \multicolumn{5}{l}{\tilde{C} d^a}\\
\end{array} \]

Thus $F \in \mathcal{C}^{0, a}\left(W \times  (0, \epsilon)\right)$.  A standard continuity argument shows that
$F \in \mathcal{C}^{0, a}\left(W\times [0, \epsilon)\right)$. 

In the second case we use the estimate that for any $0<b<1$ there exists $C'>0$ where
\[ |\log \rho|  \leq C' \rho^{-b},\]
and we repeat the same argument above.
\end{proof}

As explained in \cite{Bahuaud} we can use metric estimates to improve the regularity of the manifold transition functions via a bootstrap argument involving the transformation formula for Christoffel symbols under change of coordinates.  We have 
\begin{lemma} \label{lemma:transfncreg}
Suppose $\cA_1 = \{(U_{\alpha}, \phi_{\alpha})\}$ and $\cA_2 = \{(V_{\beta}, \psi_{\beta})\}$ are two smooth atlases arising from distinct essential subsets that are $C^{0,1}$ compatible.  Suppose that $\gbar$ is a metric that is $C^{k,\alpha}$ with respect to both atlases, for either $k=0, \alpha=1$ or $k\geq1$, $0 \leq \alpha \leq 1$.  Then $\cA_1$ and $\cA_2$ are $C^{k+1,\alpha}$ compatible.
\end{lemma}
\begin{proof}
This is a local question so we reduce to the case where $f: (U \subset \bR^{n+1}, x^i) \rightarrow (\widetilde{U} \subset \bR^{n+1}, y^i)$ is a $C^{0,1}$ diffeomorphism between open sets of $\bR^{n+1}$.  
Write the components of the metric as
\[ \gbar_{ij} = \gbar\left( \pd{}{x^i}, \pd{}{x^j} \right) \; \mbox{and} \; \gtil_{kl} = \gbar\left( \pd{}{y^k} , \pd{}{y^l} \right). \]
As $\gbar$ satisfies $C^{k,\alpha}$ estimates in both systems of coordinates, Christoffel symbols are $C^{k-1,\alpha}$ functions (bounded in the Lipschitz case).  The transformation law for Christoffel symbols under a change of coordinates states
\begin{equation} \label{eqn:ChristoffelChangeCoords}
\pd{^2 y^m}{x^i \partial x^j} = \pd{y^k}{x^i} \pd{y^l}{x^j} \Gambar^m_{kl} - \left( \Gamtil^l_{ij} \circ f \right) \pd{y^m}{x^l}. 
\end{equation}
On our first application of \eqref{eqn:ChristoffelChangeCoords} we find that the right hand side of this equation is bounded if $f \in C^{0,1}$ and if $\gbar \in C^{k,\alpha}$ with respect to both sets of coordinates.   Consequently $f$ satisfies a $C^{1,1}$ estimate.  Applying \eqref{eqn:ChristoffelChangeCoords} again with the improvement in regularity of the derivatives of $f$ allows us to conclude that the right hand side lies in $C^0$ and $f$ is consequently in $C^2$.  The rest of the argument follows by this bootstrap procedure and the fact that the product of $f$ and a $C^{0,\alpha}$ function remains $C^{0, \alpha}$ and that composition $\Gamtil^l_{ij} \circ f$ remains $C^{0, \alpha}$. 
\end{proof}

We also require an analogue of Lemma \ref{lemma:transfncreg} when the metrics enjoy only H\"older regularity.  Fortunately in this case we can use our decay estimates and Lemma \ref{lemma:holderreg} to improve the manifold regularity. 
\begin{lemma} \label{lemma:transfncreg2}
Suppose $f: (U \subset \bR^{n+1}, x^i) \rightarrow (\widetilde{U} \subset \bR^{n+1}, y^i)$ is a $C^{\infty}(U) \cap C^{0,1}(\overline{U})$ diffeomorphism between open sets of $\bR^{n+1}$.  Suppose the components of the metric are
\[ \gbar_{ij} = \gbar\left( \pd{}{x^i}, \pd{}{x^j} \right) \; \mbox{and} \; \gtil_{kl} = \gbar\left( \pd{}{y^k} , \pd{}{y^l} \right), \]
are $C^{0,\alpha}(\overline{U})$ and additionally $\partial \gbar_{ij} = O( (x^{n+1})^{a-1}), \partial \gtil_{kl} = O( (y^{n+1})^{a-1})$, for $0 < a < 1$.  Then $f \in C^{\infty}(U) \cap C^{1,\alpha}(\overline{U})$.
\end{lemma}
\begin{proof}
Again the point of departure is the formula
\begin{equation*} 
\pd{^2 y^m}{x^i \partial x^j} = \pd{y^k}{x^i} \pd{y^l}{x^j} \Gambar^m_{kl} - \left( \Gamtil^l_{ij} \circ f \right) \pd{y^m}{x^l}. 
\end{equation*}
Since $f$ is Lipschitz, all factors like $\pd{y^k}{x^i} = O(1)$.  We observe that $\Gambar^m_{kl} = O((x^{n+1})^{a-1})$.  Note that $\Gamtil^l_{ij} \circ f = O((f^{n+1})^{a-1})$, where $y^{n+1} = f^{n+1}(x^1, \ldots, x^{n+1})$ is the $(n+1)$-component function of $f$.  Since $f \in C^{0,1}(\overline{U})$, we have
\[ \left|f^{n+1}(x^1, \ldots, x^{n+1}) - f^{n+1}(x^1, \ldots, x^{n}, 0)\right| \leq C x^{n+1}, \]
where $C$ is independent of $x^1, \ldots, x^{n}$.  Consequently, $\Gamtil^l_{ij} \circ f = O((x^{n+1})^{a-1})$ and
\[ \pd{^2 y^m}{x^i \partial x^j} = O((x^{n+1})^{a-1}). \]
Lemma \ref{lemma:holderreg} now implies that $\pd{y^k}{x^i}$ are $C^{0,\alpha}$ functions.
\end{proof}

We now come to the proof of Theorems \ref{thm:A-part1} and \ref{thm:A-part2}.

\begin{proof}[Proof of Theorem \ref{thm:A-part1}]
Given an essential subset and a reference covering by truncated cylinders (cf. page \pageref{pg:refcov}), Theorem \ref{thm:comparison} gives the required estimates so that we can apply Theorem 17 of \cite{BahuaudMarsh}.  We therefore obtain that $\Mbar = M \cup M(\infty)$ is endowed with a $C^{0,1}$ structure independent of essential subset.  Given any choice of essential subset, and any choice of Fermi coordinates in the reference covering the estimates of Proposition \ref{prop:firstiteration} and Lemma \ref{lemma:holderreg} imply that the components of compactified metric $\gbar$ are $C^{0,a}$ functions if $0 < a < 1$, $C^{0,b}$ functions for every $0 < b < 1$ when $a=1$ and $C^{0,1}$ if $a > 1$.  Consequently, $\gbar$ extends to the boundary with the stated regularity.  Note that the extension remains positive definite by the metric estimate from Theorem \ref{thm:comparison}.  Consequently $g$ is conformally compact with the stated regularity.

Whenever two truncated cylinders from distinct essential subsets overlap we have a smooth transition function that is $C^{0,1}$ up to the boundary.  Since the metric enjoys H\"older/Lipschitz estimates in each cylinder, we may apply Lemma \ref{lemma:transfncreg} or \ref{lemma:transfncreg2} to improve the regularity of the transition function by one order.
\end{proof}

\begin{proof}[Proof of Theorem \ref{thm:A-part2}]
Theorem \ref{thm:A-part1} already provides the initial estimates.  We apply Proposition \ref{prop:seconditeration} and Lemma \ref{lemma:holderreg} to obtain the improvement in metric regularity, and Lemma \ref{lemma:transfncreg} to obtain the improvement in manifold regularity.
\end{proof}

% Lemma for solutions of 2nd order linear equations with asymptotically constant coefficients with exponential growth

\subsection{Appendix: the Model systems}

In this appendix we analyze the model systems.

\subsubsection{General considerations}
The model system for $0 < a < 2$ is:
\begin{equation} \label{Model.System} 
\left\{ \begin{aligned}
u' & =  ce^{-ar} u +  ce^{(2-a)r} v +  ce^{\Omega r}, \\
v' & =  ce^{-2r} u + ce^{-ar} v + b c e^{\theta r},
\end{aligned}
\right.
\end{equation}
where $b$ and $c$ are positive constants.  We do all the calculations at once; one obtains the first model system by setting $b=0$ and the second by setting $b=1$.  The model system for $a=2$ is:
\begin{equation} 
\left\{ \begin{aligned}
u' & =  cre^{-2r} u +  c v +  ce^{\Omega r}, \\
v' & =  ce^{-2r} u + cre^{-2r} v + b c e^{\theta r},
\end{aligned}
\right.
\end{equation}

We first discuss the case of \eqref{Model.System}.  

We solve for a second order equation for $v$.  Note that
\[ v' - c e^{-ar} v - bc e^{\theta r}  = c e^{-2r} u, \]
so that
\[ u = c^{-1} e^{2r} v' - e^{(2-a)r} v - b e^{(2+\theta)r} .\]
Differentiating once we obtain
\[ u' = 2 c^{-1} e^{2r} v' + c^{-1} e^{2r}  v'' - (2-a) e^{(2-a)r} v - e^{(2-a)r} v' - b (2+\theta)e^{(2+\theta)r}. \]
Substitute $u$ and $u'$ into the first equation of \eqref{Model.System} to obtain
\begin{align*}
2 c^{-1} e^{2r} v' &+ c^{-1} e^{2r} v'' - (2-a) e^{(2-a)r} v - e^{(2-a)r} v' - b (2+\theta)e^{(2+\theta)r} \\
&= ce^{-ar} (c^{-1} e^{2r} v' - e^{(2-a)r} v - b e^{(2+\theta)r})  +  c e^{(2-a)r} v + c e^{\Omega r}.
\end{align*}
This yields a second order non-homogeneous linear equation for $v$:

\begin{align*} c^{-1} e^{2r} v'' &+ (2 c^{-1} e^{2r} - 2 e^{(2-a)r)} ) v' + (- (2-a) e^{(2-a)r} + c e^{(2-2a)r} - ce^{(2-a)r} ) v \\ &= c e^{\Omega r} + b ( (2+\theta)e^{(2+\theta)r} - c e^{(2+\theta -a)r}) , 
\end{align*}
or
\begin{align*} v'' &+ (2  - 2 c e^{-ar} ) v' + (- (2-a)c e^{-ar} + c^2 e^{-2ar} -c^2 e^{-ar} ) v \\ &= c^2 e^{(\Omega - 2)r} + b ((2+\theta) c e^{\theta r} - c^2 e^{(\theta -a)r} ). \end{align*}

We repeat our calculation to solve for a second order equation for $u$.  Note that when $0 < a < 2$
\[ u' - ce^{-ar} u - ce^{\Omega r} = c e^{(2-a)r} v, \]
so that
\[ v = c^{-1} e^{(-2+a) r} u' - e^{-2r} u - e^{(\Omega-2+a)r}. \]
Differentiating we obtain
\[ v' = (-2+a)c^{-1} e^{(-2+a)r} u' + c^{-1} e^{(-2+a)r} u'' + 2 e^{-2r} u - e^{-2r} u' - (\Omega-2+a) e^{(\Omega-2+a)r}. \]
Upon substitution into the second equation of the system we obtain
\begin{align*}
(-2+a)c^{-1} e^{(-2+a)r} u' &+ c^{-1} e^{(-2+a)r} u'' + 2 e^{-2r} u - e^{-2r} u' - (\Omega-2+a) e^{(\Omega-2+a)r} \\
&= c e^{-2r} u + c e^{-ar} (c^{-1} e^{-(2-a) r} u' - e^{-2r} u - e^{(\Omega-2+a)r}) + bc e^{\theta r} .
\end{align*}
The second order equation for $u$ is then
\begin{align*}  c^{-1} e^{(-2+a)r} u'' &+ ((-2+a)c^{-1} e^{(-2+a)r} - 2e^{-2r}) u' + ( 2 e^{-2r} -c e^{-2r} + c e^{(-2-a)r} ) u \\ &= (\Omega-2+a) e^{(\Omega-2+a)r} - c e^{(\Omega-2)r} + bc e^{\theta r},
\end{align*}
which becomes
\begin{align*}  u'' &+ ((-2+a) - 2c e^{-ar}) u' + ( 2 ce^{-ar} -c^2 e^{-ar} + c^2 e^{-2ar} ) u \\ &= c(\Omega-2+a) e^{\Omega r} - c^2 e^{(\Omega-a)r} + bc^2 e^{(\theta+2-a) r}.
\end{align*}

To summarize for $a \neq 2$ we have
\begin{equation} \label{eqn:uvaneq2}
\left\{ \begin{aligned}
u'' &+ ((-2+a) - 2c e^{-ar}) u' + ( 2 ce^{-ar} -c^2 e^{-ar} + c^2 e^{-2ar} ) u \\  &= c(\Omega-2+a) e^{\Omega r} - c^2 e^{(\Omega-a)r} + bc^2 e^{(\theta+2-a) r}.\\
v'' &+ (2  - 2 c e^{-ar} ) v' + (- (2-a)c e^{-ar} + c^2 e^{-2ar} -c^2 e^{-ar} ) v \\  &= c^2 e^{(\Omega - 2)r} + b ((2+\theta) c e^{\theta r} - c^2 e^{(\theta -a)r} ).
\end{aligned}
\right.
\end{equation}

When $a=2$, these calculations yield
\begin{equation} \label{eqn:uvaeq2}
\left\{ \begin{aligned}
&u'' - 2c r e^{-2r} u' + ((-c-c^2)e^{-2r} + 2c r e^{-2r}+ c^2 r^2 e^{-4r}) u =  c\Omega e^{\Omega r} - c^2 r e^{(\Omega-2)r} + bc^2 e^{\theta r}.
\\&v'' + (2 - 2c r e^{-2r})v' + ((-c-c^2) e^{-2r} + c^2 r^2 e^{-4r} ) v= c^2 e^{(\Omega - 2)r} + b ((2+\theta) c e^{\theta r} - c^2 r e^{(\theta -2)r} ).
\end{aligned}
\right.
\end{equation}

The following propositions are our basic analytical tool for estimating generic solutions of asymptotically constant coefficient second order linear equations.  We have not presented these results in full generality to keep the statement to a reasonable size.

\begin{prop} \label{prop:asympcc} Let $a > 0$, on $[r_0, \infty)$ consider the equation
\begin{equation} \label{eqn:asympcc}
 y'' + (c_1 + e^{-ar} b_1(r)) y' + (c_2 + e^{-ar} b_2(r)) y = e^{\omega r} b_3(r), 
\end{equation}
where $c_1$, $c_2$ are constants satisfying $c_1^2 - 4c_2 > 0$ and $b_i$ are bounded smooth functions of $r$ on $[r_0, \infty]$.  Suppose $\mu_1 < \mu_2$ are distinct real roots of the characteristic polynomial for this equation.
Then all solutions to \eqref{eqn:asympcc} satisfy the following estimate:
\begin{equation*} 
 y =
\left\{ \begin{aligned}
&O(e^{\max\{\mu_2, \omega\} r}),& \mu_2 \neq \omega \\
&O(r e^{\mu_2 r}),& \mu_2 = \omega
\end{aligned}
\right.
\end{equation*}
\end{prop}
\begin{proof}
Since the coefficients of \eqref{eqn:asympcc} are asymptotically constant by \cite[Theorem 1.9.1]{Eastham} we find that two independent solutions to the associated homogeneous problem satisfy
\begin{align*}
y_1 &= (1 + o(1)) e^{\mu_1 r}, y_1' = (\mu_1 + o(1)) e^{\mu_1 r}, \mbox{and} \\
y_2 &= (1 + o(1)) e^{\mu_2 r}, y_2' = (t_2 + o(1)) e^{\mu_2 r}.
\end{align*}
We need to obtain estimates for solutions to the nonhomogeneous \eqref{eqn:asympcc}.  Recall that if $y_1, y_2$ are linearly independent solutions to an equation of the form
\[ y'' + p(r) y' + q(r) y = 0,\]
then a particular solution to the nonhomogeneous problem 
\[ y'' + p(r) y' + q(r) y = f(r),\]
is given by
\[ y_p = y_1 \int \frac{-y_2 \cdot f}{W(y_1, y_2)} + y_2 \int \frac{y_1 \cdot f}{W(y_1, y_2)}.\]
Consequently we estimate the absolute value of each of these integrals.  The Wronskian of the solutions above is asymptotic to 
\[ W(y_1, y_2) \sim (\mu_2 - \mu_1) e^{(\mu_1 + \mu_2)r}.\]
If $\omega \neq \mu_1, \mu_2$, then a simple estimation shows $y_p = O(e^{\omega r})$.  If $\omega = \mu_2$ then $y_p = O(r e^{\mu_2 r})$, whereas if $\omega = \mu_1$, $y_p$ is  $O(e^{\mu_2 r})$ by the ordering of the roots.

Consequently a generic solution $y$ to \eqref{eqn:asympcc} satisfies
\[ y = O(e^{\max\{\mu_2, \omega\} r}),  \]
when $\omega \neq \mu_2$ and when $\omega = \mu_2$
\[ y = O(r e^{\mu_2 r}).\]
\end{proof}

\begin{prop}  \label{prop:asympcc2} Let $a > 0$, on $[r_0, \infty]$ consider the equation
\begin{equation} \label{eqn:asympcc2}
 y'' + (e^{-ar} b_1(r)) y' + (e^{-ar} b_2(r)) y = e^{\omega r} b_3(r), 
\end{equation}
where $b_i$ are bounded smooth functions of $r$ on $[r_0, \infty]$.  
Then all solutions to \eqref{eqn:asympcc2} satisfy the following estimate:
\begin{equation*} 
 y = 
\left\{ \begin{aligned}
&O(r),& \omega < 0 \\
&O(r^2),& \omega = 0 \\
&O(e^{\omega r}),& \omega > 0
\end{aligned}
\right.
\end{equation*}
\end{prop}
\begin{proof}
Equation \eqref{eqn:asympcc2} has repeated characteristic roots, $\mu = 0$.  Since the coefficients of \eqref{eqn:asympcc2} are asymptotically constant by \cite[Theorem 1.10.1]{Eastham} we find that two independent solutions to the associated homogeneous problem satisfy
\begin{align*}
y_1 &= (1 + o(1)),  \; y_1' =  o(r^{-1}), \mbox{and} \\
y_2 &= (1 + o(1))r, \; y_2' = (1 + o(1)).
\end{align*}
We repeat the same analysis as before.  We find the Wronskian is asymptotically constant, and a generic solution to 
\eqref{eqn:asympcc2} satisfies
\begin{equation*} 
 y = 
\left\{ \begin{aligned}
&O(r),& \omega < 0 \\
&O(r^2),& \omega = 0 \\
&O(e^{\omega r}),& \omega > 0
\end{aligned}
\right.
\end{equation*}
\end{proof}

We now give the asymptotics for our geometric situations.  

\begin{prop} \label{prop:asypmodel1} Generic solutions of the systems
\begin{equation*} 
\left\{ \begin{aligned}
u' & =  ce^{-ar} u +  ce^{(2-a)r} v +  ce^{(3-a)r}, \\
v' & =  ce^{-2r} u + ce^{-ar} v,
\end{aligned}
\right.
\end{equation*}
for $0 < a< 2 $ and
\begin{equation*} 
\left\{ \begin{aligned}
u' & =  cre^{-2r} u +  c v +  ce^{\Omega r}, \\
v' & =  ce^{-2r} u + cre^{-2r} v,
\end{aligned}
\right.
\end{equation*}
for $a=2$ satisfy
\[ u = O(e^{(3-a)r}), \]
and
\begin{equation*} 
 v = 
\left\{ \begin{aligned}
&O(e^{(1-a)r}), & 0 < a < 1 \\
&O(r),& a = 1 \\
&O(1),& a > 1 
\end{aligned}
\right.
\end{equation*}
\end{prop}
\begin{proof}
The system is just the model system with $b=0$ and $\Omega = 3-a$.  When $a \neq 2$ we find from equations \eqref{eqn:uvaneq2} that the roots of the characteristic polynomial are $\mu = 0, 2-a$ for $u$ and $\mu = 0, -2$ for $v$.  By Proposition \ref{prop:asympcc} we find that generic solutions to these equations satisfy
\[ u = O(e^{(3-a)r}), \]
and
\begin{equation*} 
 v = 
\left\{ \begin{aligned}
&O(e^{(1-a)r}), & 0 < a < 1 \\
&O(r),& a = 1 \\
&O(1),& a > 1 
\end{aligned}
\right.
\end{equation*}

When $a=2$, the asymptotics in $v$ are unchanged but the equation for $u$ degenerates and has a repeated root $\mu = 0$.  We find the solutions are still dominated by the nonhomogeneous part and satisfy
\[ u = O(e^{(3-a)r}) = O(e^{r}).\]
\end{proof}

\begin{prop} \label{prop:asypmodel2} Generic solutions of the system
\begin{equation*} 
\left\{ \begin{aligned}
u' & =  ce^{-ar} u +  ce^{(2-a)r} v +  ce^{(4-a)r}, \\
v' & =  ce^{-2r} u + ce^{-ar} v + c e^{(1-a)r},
\end{aligned}
\right.
\end{equation*}
for $1 < a < 2$ and the system
\begin{equation*} 
\left\{ \begin{aligned}
u' & =  cre^{-2r} u +  c  v +  ce^{2r}, \\
v' & =  ce^{-2r} u + cre^{-2r} v + c e^{-r},
\end{aligned}
\right.
\end{equation*}
for $a=2$ satisfy
\[ u = O(e^{(4-a)r}), \]
and
\begin{equation*} 
 v = 
\left\{ \begin{aligned}
&O(e^{(2-a)r}), & 1 < a < 2 \\
&O(r),& a = 2
\end{aligned}
\right.
\end{equation*}
\end{prop}
\begin{proof}
The system is the model system with $b=1$ and $\Omega = 4-a$ and $\theta = 1-a$.  When $a \neq 2$ we find from equations \eqref{eqn:uvaneq2} that the roots of the characteristic polynomial are $\mu = 0, 2-a$ for $u$ and $\mu = 0, -2$ for $v$.  By Proposition \ref{prop:asympcc} we find that generic solutions to these equations satisfy
\[ u = O(e^{(4-a)r}), \]
and
\begin{equation*} 
 v = 
\left\{ \begin{aligned}
&O(e^{(2-a)r}), & 1 < a < 2 \\
&O(r),& a = 2
\end{aligned}
\right.
\end{equation*}

When $a=2$, the asymptotics in $v$ are again unchanged and
\[ u = O(e^{(4-a)r}) = O(e^{2r}).\]
\end{proof}

\section{The Einstein case}
\label{section:einstein}

In this section we specialize to the case of an Einstein metric.  We begin by introducing harmonic charts and review a theorem that gives us adequate control of the metric given \eqref{AH0} for any rate of decay $a > 0$.  Then we derive an elliptic equation for the Weyl curvature tensor and use elliptic regularity to prove Theorem \ref{thm:B}.\\

\subsection{Harmonic charts and harmonic radius}

The next two sections are dedicated to regularity questions.  In order to use elliptic regularity we need charts in which the metric is controlled.  Such charts are provided by harmonic coordinates\footnote{These charts replace the usual M\"obius coordinates used for example in \cite{Lee}.}.  These charts provided useful in the context of Einstein manifold due to the fact that the Einstein equation is elliptic.  In particular, the metric is real-analytic in such a chart (see \cite{DTK}).  Note that harmonic mapping is an important tool when addressing diffeomorphism-invariant issues because it provides a natural choice of gauge, see e.g. \cite{DeTurckRicci} in the context of Ricci Flow and \cite{Lee} in the context of Einstein equations.  Given an arbitrary smooth $(n+1)$-manifold $M$, $x \in M$, $Q > 1$, $k \in \mathbb{N}$ and $\alpha \in (0, 1)$, the $\mathcal{C}^{k,\alpha}_Q$-\textit{harmonic radius} is the largest radius $r_H=r_H(Q,k,\alpha)(x)$ such that on the geodesic ball $B_x(r_H)$ centered at $x$ with radius $r_H$, there exist harmonic coordinates in which the metric is $\mathcal{C}^{k,\alpha}_Q$-controlled :

\begin{enumerate}
\item $Q^{-1} \delta_{ij} \leq g_{ij} \leq Q \delta_{ij}$
\item $\sum_{1 \leq |\beta| \leq k} r_H^{|\beta|} \sup_x \left| \partial^\beta g_{ij}(x) \right|
		 + \sum_{|\beta|=k} r_H^{k+\alpha}\sup_{y \neq z} \frac{\left|\partial^\beta g_{ij}(y) - \partial^\beta g_{ij}(z)\right|}{d_g(y, z)^\alpha}\leq Q-1$
\end{enumerate}

We recall the following theorem from \cite{HH}:

\begin{theorem}\label{HarmRadCtrl}

Given $k \in \mathbb{N}$, $\alpha \in (0, 1)$, $Q > 1$ and $\delta > 0$. Let $(M, g)$ a smooth (n+1)-manifold without boundary and $\Omega$ an open subset of $M$. Set $\Omega_\delta = \left\{ x \in M~\text{such that}~d_g(x, \Omega) < \delta\right\}$. Assume that there exist constants $C_j,~j=0, \ldots, k$ such that:
\begin{equation}
\label{HarmRicciControl}
\left| \nabla^{(j)} \ric (x) \right| \leq C_j\quad \text{for all~}x \in M~\text{and any}~j=0, \ldots, k
\end{equation}

Assume also that the injectivity radius is bounded from below on $\Omega_\delta$:

\begin{equation}
\label{HarmIngControl}
\exists~i > 0~\text{such that}~\mathrm{inj}_{(M, g)} (x) > i\quad \forall x \in \Omega_\delta
\end{equation}

There exists a positive constant $C = C\left(n, Q, k, \alpha, \delta, i, C_1, \ldots, C_k\right)$ such that:

\begin{equation}
\label{HarmRadControl}
r_H\left(Q, k+1, \alpha\right)(x) \geq C\quad\forall x \in \Omega
\end{equation}
\end{theorem}

Note that hypothesis \eqref{HarmRicciControl} is trivially fulfilled in the case of Einstein manifold. Hypothesis \eqref{HarmIngControl} requires more work.  Set $\Omega = M \setminus K$, then $\Omega_\delta = \{r > - \delta\}$ where $r$ is to be understood as the signed distance to $\partial K$. If $\delta$ is small enough there is a diffeomorphism
$\Omega_{2\delta} \simeq (-2\delta, \infty) \times Y$ given by the normal exponential map, such that $\sec < 0$ on $\Omega_{2\delta}$ and the second fundamental form of the slices $\Sigma_r$ is positive definite. The exponential map with base point in $\Omega_\delta$ has no critical point at radius smaller than $\delta$ because of the negative curvature assumption \cite[Lemma 4.8.1]{Jost} so the injectivity radius on $\Omega_\delta$ is bounded from below if there is no closed geodesic with arbitrary small length. Even more is true:

\begin{lemma}
\label{LmClosedGeod}
There are no closed geodesics lying entirely in $\Omega_\delta$.
\end{lemma}

\begin{proof}
Let $\gamma: \mathbb{S}^1 \to \Omega_\delta$ be such a geodesic parametrized with constant speed. The function $r$ is convex on $\Omega_\delta$ because its Hessian is the second fundamental form so the image of $\gamma$ must lie in a slice $\Sigma_r$ because otherwise $r$ would reach a maximum on the image of $\gamma$. Now $\gamma$ satisfies the geodesic equation: $\nabla_{\dot{\gamma}} \dot{\gamma} = 0$ and in particular
$0=\left\langle N,\nabla_{\dot{\gamma}}\dot{\gamma}\right\rangle=-\sff\left(\dot{\gamma},\dot{\gamma}\right) \neq 0$ a contradiction.
\end{proof}

\rq This lemma can be easily extended to show that no geodesic segment $\gamma$ such that $\gamma(0) \in M \setminus \Omega_{2 \delta}$ can exit and reenter $M \setminus \Omega_{2 \delta}$.\\

This shows that hypothesis \eqref{HarmIngControl} is satisfied for an asymptotically hyperbolic manifold.  Thus for an asymptotically hyperbolic Einstein manifold, \eqref{HarmRadControl} is valid for any $k \geq 0$, $\alpha \in (0, 1)$, $Q > 1$ and any $x \in \Omega_0 = M \setminus K$.

\subsection{Asymptotic behaviour of the covariant derivatives of the Weyl tensor}

We now come to the second main result of this paper.

\begin{theorem}[Asymptotic behaviour of the covariant derivatives of the Weyl tensor]\label{WeylDerCtrl}
Assume $(M, g)$ is Einstein and satisfies \eqref{AH0} for some constant $a>0$.  There exist constants $C_j,~j=0, 1, \ldots $ such that
\begin{equation}
\left| \nabla^{(j)} \weyl \right|_g \leq C_j e^{-a r}
\end{equation}
\end{theorem}

Before diving into the proof, we recall the formula for the Laplacian of the Riemann tensor  and show how it can be transformed into an equation for the Laplacian of the Weyl tensor.  Theorem \ref{WeylDerCtrl} is then obtained by applying elliptic regularity to this equation.\\

\begin{lemma}[Laplacian of the Weyl tensor]~
Let $(M, g)$ be an arbitrary Riemannian manifold. The Laplacian of the Riemann tensor is given by:
\begin{equation}
\label{LaplRiem}
\begin{array}{rcl}
\Delta \riemdddd{a}{b}{c}{d} & = &
	- \nabla_a\nabla_c\ricdd{b}{d} + \nabla_b\nabla_c\ricdd{a}{d} + \nabla_a\nabla_d\ricdd{b}{c} - \nabla_b\nabla_d\ricdd{a}{c}\\
	& & + \ricud{j}{a} \riemdddd{j}{b}{c}{d} - \ricud{j}{b} \riemdddd{j}{a}{c}{d}
			+ 2 \left(\tdddd{B}{a}{b}{c}{d} - \tdddd{B}{a}{b}{d}{c} + \tdddd{B}{a}{c}{b}{d} - \tdddd{B}{a}{d}{b}{c}\right),
\end{array}
\end{equation}
where $$\tdddd{B}{a}{b}{c}{d} = \riemudud{i}{a}{j}{b} \riemdddd{i}{c}{j}{d}.$$

If $(M, g)$ is Einstein, i.e., $\ric_g = -ng,$ the Laplacian of the Weyl tensor is given by:
\begin{equation}
\label{LaplWeilEins}
\Delta \weyldddd{\alpha}{\beta}{\gamma}{\delta}
= -2 n  \weyldddd{\alpha}{\beta}{\gamma}{\delta} + 2 \tdddd{\tilde{Q}}{\alpha}{\beta}{\gamma}{\delta},
\end{equation}
where
\begin{eqnarray*}
\tdddd{\tilde{Q}}{a}{b}{c}{d} & = &
		\tdddd{\tilde{B}}{a}{b}{c}{d}-\tdddd{\tilde{B}}{b}{a}{c}{d}
	+ \tdddd{\tilde{B}}{a}{c}{b}{d}-\tdddd{\tilde{B}}{b}{c}{a}{d},\\
\tdddd{\tilde{B}}{a}{b}{c}{d} & = & \weyludud{i}{a}{j}{b} \weyldddd{i}{c}{j}{d}.
\end{eqnarray*}
\end{lemma}
\begin{proof} The first formula is standard (see e.g. \cite{ChowKnopf}).  We include its proof here for completeness.  The calculations are based on the second Bianchi identity:

$$\nabla_j \riemdddd{a}{b}{c}{d} + \nabla_a \riemdddd{b}{j}{c}{d} + \nabla_b \riemdddd{j}{a}{c}{d} = 0.$$

\begin{eqnarray*}
\Delta \riemdddd{a}{b}{c}{d}
		& = & g^{ij} \nabla_i \nabla_j \riemdddd{a}{b}{c}{d}\\
		& = & - g^{ij} \nabla_i \left( \nabla_a \riemdddd{b}{j}{c}{d} + \nabla_b \riemdddd{j}{a}{c}{d} \right)\\
		& = &   g^{ij} \nabla_i \left( \nabla_b \riemdddd{a}{j}{c}{d} - \nabla_a \riemdddd{b}{j}{c}{d} \right).
\end{eqnarray*}

We focus on the first term, the second one can be obtained from it by switching $a$ and $b$.

\begin{eqnarray*}
g^{ij} \nabla_i \nabla_b \riemdddd{a}{j}{c}{d}
	& = & g^{ij} \nabla_b \left( \nabla_i \riemdddd{a}{j}{c}{d} \right)\\
	& & - g^{ij}\left(\riemdddu{i}{b}{a}{k} \riemdddd{k}{j}{c}{d} + \riemdddu{i}{b}{j}{k} \riemdddd{a}{k}{c}{d}
			+ \riemdddu{i}{b}{c}{k} \riemdddd{a}{j}{k}{d} + \riemdddu{i}{b}{d}{k} \riemdddd{a}{j}{c}{k}\right)\\
	&=& - g^{ij} \nabla_b \left( \nabla_c \riemdddd{a}{j}{d}{i} + \nabla_d \riemdddd{a}{j}{i}{c} \right)
			- g^{ij} \riemddud{i}{b}{k}{a} \left(\riemdddd{k}{c}{d}{j} + \riemdddd{k}{d}{j}{c}\right)\\
	& & + \ricdu{b}{k} \riemdddd{a}{k}{c}{d}
			- \riemudud{i}{b}{k}{c} \riemdddd{i}{a}{k}{d} +\riemudud{i}{b}{k}{d} \riemdddd{a}{j}{c}{k}\\
	& = & \nabla_b \nabla_c \ricdd{a}{d} - \nabla_b \nabla_d \ricdd{a}{c} + \ricud{k}{b} \riemdddd{a}{k}{c}{d}\\
	& & + \tdddd{B}{a}{b}{c}{d} - \tdddd{B}{b}{a}{c}{d} - \tdddd{B}{a}{d}{b}{c} + \tdddd{B}{a}{c}{b}{d}.
\end{eqnarray*}
Formula \eqref{LaplRiem} now follows.\\

When $g$ is Einstein,
\begin{equation}
\riemdddd{a}{b}{c}{d} = \weyldddd{a}{b}{c}{d} - \left(g_{ad}g_{bc} - g_{ac}g_{bd}\right),
\end{equation}
so equation \eqref{LaplRiem} becomes:

$$
\Delta \riemdddd{a}{b}{c}{d} =
	-	2 n \riemdddd{a}{b}{c}{d}
	+ 2 \left(\tdddd{B}{a}{b}{c}{d} - \tdddd{B}{a}{b}{d}{c} + \tdddd{B}{a}{c}{b}{d} - \tdddd{B}{a}{d}{b}{c}\right).$$

We compute the tensor $\mathcal{B}$:

\begin{eqnarray*}
\tdddd{B}{a}{b}{c}{d}
	& = & \riemudud{i}{a}{j}{b} \left[ \weyldddd{i}{c}{j}{d} -\left(g_{id}g_{cj} - g_{ij} g_{cd}\right)\right]\\
	& = & \riemudud{i}{a}{j}{b} \weyldddd{i}{c}{j}{²d}
			- \left(\riemdddd{d}{a}{c}{b} + \ricdd{a}{b} g_{cd}\right)\\
	& = & \weyludud{i}{a}{j}{b} \weyldddd{i}{c}{j}{d}
			- \left(  \kronecker{i}{b} \kronecker{j}{a} - g^{ij} g_{ab}\right) \weyldddd{i}{c}{j}{d}
			+ n g_{ab} g_{cd} - \riemdddd{a}{d}{b}{c}\\
	& = & \tdddd{\tilde{B}}{a}{b}{c}{d} - \weyldddd{b}{c}{a}{d}	+ n g_{ab} g_{cd}
			 - \left[\weyldddd{a}{d}{b}{c} - \left(g_{ac} g_{db} - g_{ab} g_{dc}\right) \right]\\
	& = & \tdddd{\tilde{B}}{a}{b}{c}{d} - 2  \weyldddd{a}{d}{b}{c}	+ n g_{ab} g_{cd}
			 + \left(g_{ac} g_{db} - g_{ab} g_{dc}\right)\\
	& = & \tdddd{\tilde{B}}{a}{b}{c}{d} - 2 \weyldddd{a}{d}{b}{c}	+  (n-1) g_{ab} g_{cd}
			 +  g_{ac} g_{db}.\\
\end{eqnarray*}
Thus

\begin{eqnarray*}
\tdddd{B}{a}{b}{c}{d} + \tdddd{B}{a}{c}{b}{d}
&	= & \tdddd{\tilde{B}}{a}{b}{c}{d} + \tdddd{\tilde{B}}{a}{c}{b}{d}
	- 2  \left(\weyldddd{a}{d}{b}{c} + \weyldddd{a}{d}{c}{b}\right)
	+ n  \left( g_{ab} g_{cd} + g_{ac} g_{db}\right)\\
& = & \tdddd{\tilde{B}}{a}{b}{c}{d} + \tdddd{\tilde{B}}{a}{c}{b}{d}
	+ n  \left( g_{ab} g_{cd} + g_{ac} g_{db}\right),
\end{eqnarray*}
so the quadratic term in \eqref{LaplRiem} can be written

\begin{eqnarray*}
\tdddd{Q}{a}{b}{c}{d}
	& = & \tdddd{B}{a}{b}{c}{d} + \tdddd{B}{a}{c}{b}{d} - \tdddd{B}{a}{b}{d}{c} - \tdddd{B}{a}{d}{b}{c}\\
	& = & \tdddd{\tilde{Q}}{a}{b}{c}{d}	+ n \left(g_{ac} g_{db} - g_{ad} g_{cb}\right).\\
\end{eqnarray*}
Upon plugging this expression into the expression of the Laplacian of the Riemann tensor, we obtain:

\begin{eqnarray*}
\Delta \weyldddd{a}{b}{c}{d}
		& = & \Delta \riemdddd{a}{b}{c}{d}\\
		& = & 2 \left(\tdddd{B}{a}{b}{c}{d} - \tdddd{B}{a}{b}{d}{c} + \tdddd{B}{a}{c}{b}{d} - \tdddd{B}{a}{d}{b}{c}\right)
				- 2 n  \riemdddd{a}{b}{c}{d}\\
		& = & 2 \tdddd{\tilde{Q}}{a}{b}{c}{d} + 2n  \left(g_{ac} g_{bd} - g_{ad} g_{bc}\right)
				- 2 n  \riemdddd{a}{b}{c}{d}\\
		& = & 2 \tdddd{\tilde{Q}}{a}{b}{c}{d} - 2 n  \weyldddd{a}{b}{c}{d}.
\end{eqnarray*}
\end{proof}

\begin{proof}[Proof of Theorem \ref{WeylDerCtrl}]

Select $Q > 1$, $\alpha \in (0,~1)$, and $p > n+1$ (the values of these constants do not influence the proof), set $r_H=r_H(Q,k+2,\alpha)$ the $(Q,k+2,\alpha)$-harmonic radius of $(M, g)$.  By the discussion following Theorem \ref{HarmRadCtrl}, we have that $r_H\geq C(n, Q, k, \alpha)>0$. Let $x \in \Omega$, choose harmonic coordinates $z^i$ on $B_x(r_H)$ such that $g$ is $\mathcal{C}^{k+2,\alpha}_Q$-controlled.  In this chart, equation \eqref{LaplWeilEins} can be written:

$$ g^{ij} \partial_i \partial_j \weyl + \Gamma * \partial \weyl + \partial\Gamma * \weyl + \Gamma \Gamma * \weyl + 2n \weyl = \tilde{\mathcal{Q}}.$$

The $\Gamma$ terms are easily estimated: $\Gamma^k_{ij} = \frac{1}{2} g^{kl} \left(\partial_i g_{jl} + \partial_j g_{il} - \partial_l g_{ij}\right)$.  So, in the chart $z^i$, $\|\Gamma\|_{k+1, \alpha} \leq f_k \left(\|g_{ij}\|_{k+2, \alpha}\right)$ where $f_k$ is a polynomial function depending only on $k$.  Applying the interior Schauder estimate we obtain

$$
\left\|\weyl\right\|_{W^{k+2, p}\left(B\left(\frac{r_H}{2}\right)\right)}
\leq C_{Q, k, p} \left( \left\|\tilde{\mathcal{Q}}\right\|_{W^{k, p}\left(B\left(r_H\right)\right)}
+ \left\|\weyl\right\|_{L^p\left(B\left(r_H\right)\right)}\right).
$$

With our assumptions, $W^{k, p}\left(B\left(r_H\right)\right)$ is a Banach algebra (\cite[Theorem 5.23]{Adams}) so\footnote{Note that, because of contractions in the expression of $Q$, this estimate relies once more on the fact that the metric is $\mathcal{C}^{k+2,\alpha}_Q$-controlled}:

$$
\left\|\weyl\right\|_{W^{k+2, p}\left(B\left(\frac{r_H}{2}\right)\right)}
\leq \tilde{C}_{Q, k, p} \left( \left\|\weyl\right\|^2_{W^{k, p}\left(B\left(r_H\right)\right)}
+ \left\|\weyl\right\|_{L^p\left(B\left(r_H\right)\right)}\right).
$$

A simple induction argument over $k$ gives the following estimate: assume that $\left\|\weyl\right\|_{L^p\left(B\left(\frac{r_H}{2}\right)\right)} \leq C_0 e^{-a r_0}$ then there exists a constant $C_{k+2}$ such that $\left\|\weyl\right\|_{W^{k+2,p}\left(B\left(\frac{r_H}{2^{\lceil\frac{k}{2}\rceil}}\right)\right)}\leq C_{k+2} e^{-a r_0}$.  The Sobolev embedding Theorem leads to the estimate $\left\|\weyl\right\|_{\mathcal{C}^{k+1}\left(B\left(\frac{r_H}{2}\right)\right)} \leq C_{k+2} e^{-a r_0}$. The $\mathcal{C}^{k+1,\alpha}$-control on the Christoffel symbols allows us to replace the derivatives by covariant derivatives:

$$\sum_{j=0}^{k+1}\sup_{y\in B_x\left(\frac{r_H}{2}\right)}\left|\nabla^{(j)}\weyl(y)\right|\leq\tilde{C}_{k+2}e^{-a r_0}.$$

In particular for $z = x$:

$$\sum_{j=0}^{k+1} \left| \nabla^{(j)} \weyl (x) \right| \leq \tilde{C}_{k+2} e^{-a r}.$$
\end{proof}

\bibliographystyle{amsplain}
\bibliography{ahe}
\end{document}